\UseRawInputEncoding

\RequirePackage{amsthm}
\documentclass[sn-mathphys,Numbered]{sn-jnl}

\usepackage{subfigure}
\usepackage{graphicx}%
\usepackage{multirow}%
\usepackage{amsmath,amssymb,amsfonts}%

\usepackage{mathrsfs}%
\usepackage[title]{appendix}%
\usepackage{xcolor}%
\usepackage{textcomp}%
\usepackage{manyfoot}%
\usepackage{booktabs}%
\usepackage{algorithm}%
\usepackage{algorithmicx}%
\usepackage{algpseudocode}%
\usepackage{listings}%

\usepackage{bookmark}

\theoremstyle{thmstyleone}%
\newtheorem{theorem}{Theorem}
\theoremstyle{thmstyletwo}%
\newtheorem{example}{Example}%
\newtheorem{remark}{Remark}%

\theoremstyle{thmstylethree}%
\newtheorem{definition}{Definition}%
\newtheorem{lemma}{Lemma}
\newcommand{\st}{\mbox{subject to}}

\begin{document}

\title[Quantile-based Random Sparse Kaczmarz for Corrupted and Noisy Linear Systems]{Quantile-based Random Sparse Kaczmarz for Corrupted and Noisy Linear Systems}


\author[1]{\fnm{Lu} \sur{Zhang}}\email{zhanglu21@nudt.edu.cn}

\author[1]{\fnm{Hongxia} \sur{Wang}}\email{wanghongxia@nudt.edu.cn}

\author*[1]{\fnm{Hui} \sur{Zhang}}\email{h.zhang1984@163.com}

\affil[1]{\orgdiv{Department of Mathematics}, \orgname{National University of Defense Technology}, \orgaddress{\city{Changsha}, \postcode{410073}, \state{Hunan}, \country{China}}}


\abstract{The randomized Kaczmarz method, along with its recently developed variants, has become a popular tool for dealing with large-scale linear systems. However, these methods usually fail to converge when the linear systems are affected by heavy corruptions, which are common in many practical applications. In this study, we develop a new variant of the randomized sparse Kaczmarz method with linear convergence guarantees, by making use of a quantile technique to detect corruptions. Moreover, we incorporate the averaged block technique into the proposed method to achieve parallel computation and acceleration. Finally, the proposed algorithms are illustrated to be very efficient through extensive numerical experiments.}




\maketitle
\section{Introduction}\label{sec1}

Recently, there has tremendous interest in developing Kaczmarz-type methods for large-scale linear systems in many applied fields, such as compressed sensing \cite{lorenz2014sparse}, phase retrieval \cite{tan2019phase,xian2022randomized,romer2021randomized}, tensor recovery \cite{chen2021regularized,du2021randomized}, medical imaging \cite{gordon1970algebraic} and so on.
The main challenge is to find certain solutions in an efficient way for linear systems with large-scale data, which may
be damaged by random noise due to imprecision processing or partially corrupted by data collection, transmission, adversarial components, and modern storage systems \cite{jarman2021quantilerk}. In particular, the corrupted data is usually useless and even harmful, but not easy to detect. Under some assumptions, the linear system may be formulated as
\begin{eqnarray}
\label{eq:1}
Ax=b,
\end{eqnarray}
where $A\in \mathbb{R}^{m\times n}$ is a matrix used to model the linear measurement procedure and $b\in \mathbb{R}^m$ is the measurement data. If we let $\hat{x}$ stand for the objective that we want to recover from the measurement data, then $b$ may not exactly equal to $A\hat{x}$ due to the possible existence of random noise and corruptions. In order to connect $b$ and $A\hat{x}$, we let $\tilde{b}=A\hat{x}$; then the measurement data $b$ can be expressed as
\begin{eqnarray}
b=\tilde{b}+b^c+r,\nonumber
\end{eqnarray}
where $b^c\in\mathbb{R}^m$ represents the fraction of corruptions and $r\in \mathbb{R}^m$ records random noise. In the case of $b^c=0$, the random Kaczmarz method has proven to be highly efficient for solving linear system \eqref{eq:1}; see for example \cite{needell2010randomized,schopfer2019linear,yuan2022sparse,yuan2022adaptively,zhang2022weighted}.
Recently, with the help of quantile technique, the authors of \cite{haddock2022quantile,jarman2021quantilerk,steinerberger2023quantile} have made effort to extend the random Kaczmarz method to deal with the case of $b^c$ being sparse in the sense that only a few entries of $b^c$ are nonzero but these entries may be arbitrarily large. However, the quantile-based random Kaczmarz (Quantile-RK) method in \cite{haddock2022quantile} is currently limited to finding least squares solutions to corrupted linear systems; on the other hand, the random sparse Kaczmarz (RaSK) method in \cite{schopfer2019linear} is only limited to find sparse solutions to uncorrupted linear system. A natural question is whether we can take advantage of both of them to design a new algorithm so that we are able to find sparse or least squares solutions to possibly corrupted linear systems?

In this paper, we give an affirmative answer to the question above by proposing a new iterative solver, called sparse quantile-based Kaczmarz (Quantile-RaSK). The proposed method inherits the advantages of RaSK and Quantile-RK and 
is functionally more powerful than both of them. Besides, it can also be viewed as a general algorithmic framework that includes many existing Kaczmarz-type methods as its special cases. Theoretically, we show that Quantile-RaSK converges in expectation with respect to the Bregman distance with linear convergence rates, which involves estimations of singular values of the measurement matrix $A$. In particular, we consider a class of subgaussian random matrices whose singular values are well studied, and hence the linear convergence rates can be reformulated.

In order to achieve parallel computation and fast convergence, we further propose a block version of the Quantile-RaSK method, abbreviated as Quantile-RaSKA, by incorporating the averaged block technique in \cite{tondji2023faster} and the extrapolation technique in \cite{merzlyakov1963relaxation,necoara2019faster}. Also, we show that Quantile-RaSKA converges linearly with rates related to singular values of the measurement matrix $A$.  By specifying the case of subgaussian random matrices, we show that Quantile-RaSKA converges faster than Quantile-RaSK by a linear factor.

\subsection{Notations}
We use $[\cdot]$ for the integral function. Let $\beta$ be the fraction of corruptions and $q$ be the proportion of acceptable hyperplanes.
For simplicity, we assume that $\beta m$ and $qm$ are integers instead of using $[\beta m]$ and $[qm]$.
For a sequence $y_i,1\leq i\leq n$, which is sorted from small to large as $y_{(i)},1\leq i\leq n$,
the $q\textrm{-quantile}$ of $y_i,1\leq i\leq n$ is defined by
\begin{eqnarray}		
y_q:= \left \{		
\begin{array}{lr}
y_{([nq]+1)},& \textrm{if $nq$~is~not~an integer},\\			
(y_{(nq)}+y_{(nq)+1})/2,& \textrm{otherwise}.		
\end{array}		
\right.		
\end{eqnarray}	
Define the $q$-quantile of residuals indexed in $S$ as
\begin{equation}
Q_q(x,S):={q\textrm{-quantile}}(|\langle x_k,a_i\rangle-b_i|:i\in S).\nonumber
\end{equation}
Moreover, we denote $\mathbb{E}_k$ as the conditional expectation with respect to the random sample selected in the $k$-th under the condition of the $k-1$ previous iterations.
For arbitrary set $C$,
$|C|$ represents the number of elements in set $C$.
If without a declaration, then norm $\|\cdot\|$ refers to the Euclidean norm.

\subsection{Related work}
\subsubsection{Randomized sparse Kaczmarz}
The standard Kaczmarz method \cite{karczmarz1937angenaherte} was proposed by Kaczmarz in 1937, which approximates the solution of linear systems by orthogonal projections onto the hyperplanes $a_{i_k}^Tx=b_{i_k}$ cyclically, with $a_{i_k}$ being the $i_k$-th row of $A$ and $b_i$ being the $i$-th entry of $b$. Let $i_k=\text{mod}(k-1,m)+1$; then the cyclical projection can be formulated as 
\begin{equation}
x_{k+1}=x_{k}-\frac{\langle a_{i_k},x_k\rangle-b_{i_k}}{\|a_{i_k}\|^2}a_{i_k}.\nonumber
\end{equation}
Later, the Kaczmarz method was rediscovered as the Algebraic Reconstruction Technique (ART) in computerized tomography \cite{hounsfield1973computerized}. Lots of effort was made to obtain quantitative convergence rates; see for example \cite{neumann1950functional,halperin1962product,deutsch1997rate,galantai2005rate}. Recently, the randomized Kaczmarz (RK) method was analyzed in \cite{strohmer2009randomized}, deducing a surprisingly simple convergence rate that only depends on the smallest singular value and the Frobenius norm of $A$. In  \cite{lorenz2014linearized,lorenz2014sparse,schopfer2019linear}, the randomized sparse Kaczmarz (RaSK) method was proposed to produce sparse solutions; it reads as 
\begin{align}
x_{k+1}^*&=x_k^*-\frac{\langle a_{i_k},x_k\rangle-b_{i_k}}{\|a_{i_k}\|^2}\cdot a_{i_k},\nonumber\\
x_{k+1}&=\mathcal{S}_{\lambda}(x_{k+1}^*),\nonumber
\end{align}
with $\lambda>0$ and soft shrinkage function $\mathcal{S}_{\lambda}(x):=\max\{|x|-\lambda,0\}\cdot \textrm{sign}(x)$.
For a consistent linear system $Ax=b$ the iterates $x_k$ converge to the unique sparse solution $\hat{x}$ of the strongly convex regularized basis pursuit problem \cite{chen2001atomic,cai2009convergence}:
\begin{equation}
\label{eq:1.3}
\min_{x\in \mathbb{R}^n} f(x):=\lambda \|x\|_1+\frac{1}{2}\|x\|^2, \st~ Ax=b.
\end{equation}
For accelerated variants of RaSK, the reader may refer to \cite{yuan2022sparse,yuan2022adaptively,zhang2022weighted,petra2015randomized,jiang2020kaczmarz}.
\subsubsection{Block sparse Kaczmarz}
In block sparse Kaczmarz methods, a subsystem $A_{P_k}x=b_{P_k}$ is used instead of a single hyperplane $a_{i_k}^Tx=b_{i_k}$ in each iteration, where $P_k\subset [m]$ and $|P_k|\geq 1$. Block sparse Kaczmarz methods may be divided into two types. The first type is the projective block sparse Kaczmarz method, whose iterate scheme is given by
\begin{equation}
\begin{aligned}
\label{eq:1.5}
x_{k+1}^*&=x_k^*-w_{k}\frac{A_{P_k}^T(A_{P_k}x_k-b_{P_k})}{\|A_{P_k}\|_2^2}\\
x_{k+1}&=\mathcal{S}_{\lambda}(x_{k+1}^*),
\end{aligned}
\end{equation}
where $w_k$ represents the extrapolated stepsize or relaxed parameter, used to enhance the practical performance of the block sparse Kaczmarz method. The extrapolated stepsize may be constant, adaptive stepsize, and Chebychev-based stepsize, etc \cite{necoara2019faster}.
The first type of block sparse Kaczmarz methods are considered in
\cite{necoara2019faster,needell2014paved} (with $\lambda=0$) and
\cite{tondji2023faster,petra2015randomized} (with $\lambda>0$). The second type is the averaged block sparse Kaczmarz method, reading as
\begin{equation}
\begin{aligned}
\label{eq:1.6}
x_{k+1}^*&=x_k^*-\frac{1}{|P_k|} \sum_{i\in P_k}w_i\frac{\langle a_i,x_k\rangle-b_i}{\|a_i\|_2^2}\cdot a_i,\\
x_{k+1}&=\mathcal{S}_{\lambda}(x_{k+1}^*).
\end{aligned}
\end{equation}
Note that the averaged block method makes use of the advantage of distributed computing. Such averaged block Kaczmarz-type methods are considered in \cite{necoara2019faster,moorman2021randomized,miao2022greedy} (with $\lambda=0$) and \cite{tondji2023faster} (with $\lambda>0$).
When $|P_k|=1$, both (\ref{eq:1.5}) and (\ref{eq:1.6}) reduce to the randomized sparse Kaczmarz method \cite{schopfer2019linear}; when $|P_k|=m$, the projective block sparse Kaczmarz method (\ref{eq:1.5}) reduces to the linearized Bregman method \cite{lorenz2014linearized,cai2009convergence,yin2010analysis}.

\subsubsection{Quantile randomized sparse Kaczmarz}
Suppose that we have the desired objective $\hat{x}$ at hand. Then it is easy to detect the corrupted data by considering the residual $\langle a_i,\hat{x}\rangle-b_i$ since a large residual implies corruption. This simple observation can help us develop heuristic methods. In fact, the multiple round Kaczmarz with removal method proposed in \cite{haddock2019randomized} is based on the intuition that the entries $b_i$ with large residuals are likely to be corrupted.
However, the number of corrupted entries must be very small in order to recover the desired solution $\hat{x}$ \cite{haddock2019randomized}. By borrowing the quantial technique, the authors of \cite{haddock2022quantile} proposed a quantile-based randomized Kaczmarz (Quantile-RK) method, which is still efficient when the corruption fraction is high. Very recently, a variant of Quantile-RK was proposed in \cite{cheng2022block}, called quantile-based averaged block Kaczmarz (Quantile-RKA), by incorporating the averaged block Kaczmarz method in \cite{necoara2019faster}.
The numerical experiment in Section 4.3 of the paper \cite{cheng2022block} shows that Quantile-RK with projective block may fail to converge even if there is only one single corrupted row in each block; however, Quantile-RK with averaged block continues to behave well. Compared with Quantile-RK with projective block, Quantile-RK with averaged block can reduce the effect of selected corrupted hyperplanes on the iterate. Therefore, we choose the averaged block technique to help us design new and accelerated iterative schemes.

\subsection{Contribution and organization}
We highlight the main contributions as follows:
\begin{itemize}
	\item propose quantile-based randomized sparse Kaczmarz (Quantile-RaSK) for finding sparse solutions to corrupted and noisy linear systems;
	\item propose quantile-based randomized sparse Kaczmarz with averaged block (Quantile-RaSKA) for parallel computation and acceleration;
	\item deduce detailed rates of linear convergence for the proposed methods.
\end{itemize}

The remainder of the paper is organized as follows. In Section 2, we introduce basic tools for analysis. In Section 3, the Quantile-RaSK method is proposed formally, along with linear convergence analysis. In Section 4,  the Quantile-RaSKA method based on averaged block technique is given; moreover, we verify that Quantile-RaSKA outperforms Quantile-RaSK in the subgaussian case.
In Section 5, we carry out numerical experiments to compare the performance of Quantile-RaSK and Quantile-RaSKA. In Section 6, we briefly remark the conclusions.

\section{Preliminaries}
First, we recall some basic tools for convex analysis. 
Let $f:\mathbb{R}^n\rightarrow \mathbb{R}$ be a convex function. The subdifferential of $f$ at $x\in \mathbb{R}^n$ is denoted by
\begin{eqnarray}
\partial f(x):=\{x^*\in \mathbb{R}^n|f(y)\geq f(x)+\langle x^*,y-x\rangle, \forall y\in \mathbb{R}^n\}.\nonumber
\end{eqnarray}
If the convex function $f$ is assumed to be differentiable, then $\partial f(x)=\{\nabla f(x)\}.$
We say that a convex function $f:\mathbb{R}^n\rightarrow \mathbb{R}$ is $\alpha$-strongly convex if there exits $\alpha>0$ such that for $\forall x,y\in \mathbb{R}^n$ and $x^*\in \partial f(x)$, we have
\begin{eqnarray}
f(y)\geq f(x)+\langle x^*,y-x\rangle+\frac{\alpha}{2}\|y-x\|^2.\nonumber
\end{eqnarray}
It can be easily verified that $f(\cdot)+\frac{\lambda }{2}\|\cdot\|^2$ is $\lambda$-strongly convex if $\lambda>0$ for any convex function $f$.
The Fenchel conjugate $f^*$ of $f$ is given by
\begin{eqnarray}
f^*(x):=\sup_{z\in\mathbb{R}^{n}}\{\langle x,z\rangle-f(z)\}.\nonumber
\end{eqnarray}

\subsection{Bregman distance and projection}
\begin{definition}[\cite{lorenz2014linearized}]
	Let $f:\mathbb{R}^n\rightarrow \mathbb{R}$ be a convex function. The Bregman distance between $x,y \in \mathbb{R}^n$ with respect to f and $x^*\in \partial f(x)$ is defined as
	\begin{eqnarray}
	D_f^{x^*}(x,y):=f(y)-f(x)-\langle x^*,y-x\rangle.
	\nonumber
	\end{eqnarray}
\end{definition}
For $f(x)=\frac{1}{2}\|x\|^2$, we have $ D_f^{x^*}(x,y)=\frac{1}{2}\|y-x\|^2.$ The following is an example of an nonsmooth but strongly convex function $f$.
\begin{example}
	\label{example:1}
	For $f(x)=\lambda \|x\|_1+\frac{1}{2}\|x\|^2$, its subdifferential is given by
	$\partial f(x)=\{x+\lambda\cdot s|s_i=\textrm{sign}(x_i),x_i\neq 0\ and\ s_i\in[-1,1],x_i=0\}$. We have $x^*=x+\lambda\cdot s\in\partial f(x)$ and  $D_f^{x^*}(x,y)=\frac{1}{2}\|y-x\|^2+\lambda(\|y\|_1-\langle s,y\rangle)$.
\end{example}
For strongly convex functions,  the Bregman distance can be bounded by the Euclidean distance. These bounds will be used in the convergence analysis of Quantile-RaSK.
\begin{lemma}[\cite{schopfer2019linear}]\label{lemma:2}
	Let $f:\mathbb{R}^n\rightarrow \mathbb{R}$ be $\alpha$-strongly convex. Then for all $x,y\in \mathbb{R}^n$ and $x^*\in \partial f(x), y^*\in \partial f(y)$, we have
	$$
	\frac{\alpha}{2}\|x-y\|^2\leq D_f^{x^*}(x,y)\leq \langle x^*-y^*,x-y\rangle\leq \|x^*-y^*\|\cdot\|x-y\|,
	$$
	and hence
	$D_f^{x^*}(x,y)=0$ if and only if $x=y$.
\end{lemma}

\begin{definition}[\cite{lorenz2014linearized}]
	Let $f:\mathbb{R}^n\rightarrow \mathbb{R}$ be $\alpha$-strongly convex and $C\subset \mathbb{R}^n$ be a nonempty closed convex set. The Bregman projection of x onto C with respect to f and $ x^*\in \partial f(x)$ is the unique point $\Pi_{C}^{x^*}(x)\in C$ such that
	\begin{eqnarray}
	\Pi_{C}^{x^*}(x):=\arg\min_{y\in C}D^{x^*}_f(x,y).
	\end{eqnarray}
\end{definition}

In general, it is hard to compute the Bregman projection since we need to solve a constrained nonsmooth convex optimization. The next lemma presents a way to obtain the Bregman projection onto affine subspaces.
\begin{lemma}[\cite{lorenz2014linearized}]
	\label{lemma2.3}
	Let $f:\mathbb{R}^n\rightarrow \mathbb{R}$ be $\alpha$-strongly convex, $a\in \mathbb{R}^n$, $\gamma\in \mathbb{R}.$
	The Bregman projection of $x\in \mathbb{R}^n$ onto the hyperplane $H(a,b)=\{x\in \mathbb{R}^n|\langle a,x\rangle=b\}$ with $a\neq 0$ is
	\begin{eqnarray}
	z:=\Pi_{H(a,b)}^{x^*}(x)=\nabla f^*(x^*-\hat{t}\cdot a),
	\nonumber
	\end{eqnarray}
	where $\hat{t}\in \mathbb{R}$ is a solution of
	$\min_{t\in \mathbb{R}}f^*(x^*-t\cdot a)+t\cdot b.$
	Moreover, for $z^*=x^*-\hat{t}a$ and all $y\in H(a,b)$ we have
	\begin{eqnarray}\label{lemma3:4}
	D^{z^*}_f(z,y)\leq D^{x^*}_f(x,y)-\frac{\alpha}{2}\frac{(\langle a,x\rangle-b)^2}{\|a\|^2}.
	\end{eqnarray}
	If $x\notin H_{\leq (a,b)}:=\{x\in \mathbb{R}^n|\langle a,x\rangle \leq b\}$,
	then
	$\Pi_{H_{\leq (a,b)}}^{x^*}(x)=\Pi_{H{ (a,b)}}^{x^*}(x)$.
\end{lemma}

\subsection{Minimum singular value constants}
Our convergence analysis will heavily rely on the following constants; similar concepts previously appeared in \cite{steinerberger2023quantile,haddock2022quantile}.
The smallest and largest singular values of $A$ are denoted by $\sigma _{\min}(A)$ and $\sigma _{\max}(A)$ respectively, given by
\begin{equation}
\sigma_{\min }(A):=\min_{x\in \mathbb{R}^n,x\neq 0}\frac{\|Ax\|}{\|x\|},
\sigma_{\max }(A):=\max_{x\in \mathbb{R}^n,x\neq 0}\frac{\|Ax\|}{\|x\|}.\nonumber
\end{equation}
Denote the matrix that is formed by the columns of $A$ indexed by $J$ as $A_J$, and then define
\begin{align}
\tilde{\sigma}_{\min }(A):=\min \left\{\sigma_{\min }\left(A_J\right) \mid J \subset\{1, \ldots, n\}, A_J \neq 0\right\},\nonumber
\end{align}
Let $A_{I,J}$ be the submatrix consisting of the rows of $A$ indexed by $I$ and the columns of $A$ indexed by $J$. Define
$$
\label{2:1}
~~~\tilde{\sigma}_{q-\beta,\min}(A):
=\min\{\sigma_{\min}(A_{I,J})~:~ |I|=(q-\beta)m,I\subset
\{1,\dots,m\},J\subset
\{1,\dots,n\}\}.
$$
It should be noted that our defined constant $\tilde{\sigma}_{q-\beta,\min}(A)$ is slightly different from the one in \cite{steinerberger2023quantile,haddock2022quantile}, that is, 
\begin{eqnarray}
\sigma_{q-\beta,\min}(A):=\min\{\sigma_{\min}(A_{I})~:~ |I|=(q-\beta)m,I\subset
\{1,\dots,m\}\},\nonumber
\end{eqnarray}
which only involves the submatrices of $(q-\beta)m$ rows about $A$. For simplicity, we respectively rewrite them as $\sigma_{\min },\sigma_{\max },\tilde{\sigma}_{\min}, \tilde{\sigma}_{q-\beta,\min}$ and $\sigma_{q-\beta,\min}$ in this paper.
We remark that these constants are hard to estimate in general. 

\subsection{Convergence analysis of RaSK}
Here, we recall the convergence results of RaSK in noiseless and noisy cases respectively \cite{schopfer2019linear}, which will be used in our convergence analysis.
Let $ \textrm{supp}(\hat{x})=\left\{j \in\{1, \ldots, n\} \mid \hat{x}_j \neq 0\right\}$ and  $\tilde{\kappa}=\frac{\|A\|_F}{\tilde{\sigma}_{\min }}$. When $b \neq 0$ we also have $\hat{x} \neq 0$, and hence
\begin{align}
|\hat{x}|_{\min }:=\min \left\{\left|\hat{x}_j\right| \mid j \in \textrm{supp}(\hat{x})\right\}>0.\nonumber
\end{align}

\begin{lemma}[\cite{schopfer2019linear}, Theorem 3.2]
	The iterates $x_k$ of the RaSK method with inexact or exact steps converge in expectation to the true
	solution with a linear rate, it holds that
	\begin{equation}
	\label{convergence:RaSK}
	\mathbb{E}\left(D_f^{x_{k+1}^*}(x_{k+1},\hat{x})\right)
	\leq\left(1-\frac{1}{2}\cdot\frac{1}{\tilde{\kappa}^2} \cdot \frac{|\hat{x}|_{\min }}{|\hat{x}|_{\min }+2 \lambda}\right)\cdot \mathbb{E}\left(D_f^{x_{k}^*}(x_{k},\hat{x})\right).
	\end{equation}
\end{lemma}

\begin{lemma}[\cite{schopfer2019linear}, Theorem 3.4]
	\label{lemma:rask}
	Assume that instead of exact data $b$ only a noisy right hand side $b^{\delta}\in\mathbb{R}^m$ with $\|b^{\delta}-b\|\leq \delta$ is given, we have\\
	(a) for the RaSK with inexact step:
	\begin{align}
	\label{rasknoisy:1}
	\mathbb{E}(D_f^{x_{k+1}^*}(x_{k+1},\hat{x}))\leq
	\left(1-\frac{1}{2}\cdot\frac{1}{\tilde{\kappa}^2} \cdot\frac{|\hat{x}|_{\min}}{|\hat{x}|_{\min}+2\lambda}\right)
	\mathbb{E}(D_f^{x_{k}^*}(x_{k},\hat{x}))
	+\frac{1}{2}\frac{\|\delta\|^2}{\|A\|_F^2};
	\end{align}
	(b) for the RaSK with exact step (abbreviated by ERaSK):
	\begin{align}
	\label{rasknoisy:2}
	&~~~~\mathbb{E}(D_f^{x_{k+1}^*}(x_{k+1},\hat{x}))\nonumber\\
	&\leq
	\left(1-\frac{1}{2}\cdot\frac{1}{\tilde{\kappa}^2} \cdot\frac{|\hat{x}|_{\min}}{|\hat{x}|_{\min}+2\lambda}\right)
	\mathbb{E}(D_f^{x_{k}^*}(x_{k},\hat{x}))
	+\frac{1}{2}\frac{\|\delta\|^2}{\|A\|_F^2}
	+
	\frac{2\|\delta\|\cdot\|A\|_{1,2}}{\|A\|_F^2},
	\end{align}
	where $\|A\|_{1,2}=\sqrt{\sum_{i=1}^m \|a_i\|_1^2}$. Let $\alpha=\frac{|\hat{x}|_{\min}}{|\hat{x}|_{\min}+2\lambda}$.
\end{lemma}

\section{The Quantile-RaSK method}
In this section, we first propose the quantile-based randomized sparse Kaczmarz method, abbreviated by Quantile-RaSK. Then, we present a group of convergence results in noisy and noiseless cases under the existence of $\beta m$ corruptions; all proofs are deferred to Appendix.
\subsection{The proposed Quantile-RaSK method}
Motivated by the Quantile-RK method proposed in \cite{haddock2022quantile,jarman2021quantilerk,steinerberger2023quantile}, we employ an important statistical variable, quantile, to detect the corrupted measurements.
Under the assumption of standardized rows of $A$, the absolute residual $|a_{i_k}^Tx_k-b_{i_k}|$ is the distance from the iterate $x_k$ to the hyperplane $H(a_{i_k},b_{i_k})$.
Concretely, we first compute the $q$-quantile of residuals in each iterate, treat a row whose residuals are less than the quantile as an acceptable row, and then project it onto an acceptable hyperplane obtained by uniform sampling.
In order to recover sparse solutions, we use the nonsmooth Bregman projection with an augmented $\ell_1$ norm. In Lemma \ref{lemma2.3}, the computation of the Bregman projection needs to solve the minimization problem
\begin{eqnarray}
t_k:=\arg\min_{t\in \mathbb{R}}f^*(x_k^*-ta_{i_k})+tb_{i_k},\nonumber
\end{eqnarray}
which may not be exactly solved. As recommended by the authors in \cite{lorenz2014linearized}, we may choose an inexact step $t_k=\langle a_{i_k},x_k\rangle-b_{i_k}$ with a fixed index $i_k$, which has been chosen in a random way before computing the step $t_k$.
For generality, all rows $a_i^T,1\leq i\leq m$ are normalized.
In detail, Quantile-RaSK with exact step is denoted by Quantile-ERaSK.
The pseudocode of Quantile-RaSK with inexact step and exact step is presented as follows.

When the sampled block unfortunately contains a corrupted hyperplane, the iterate will be far away from the true solution of the uncorrupted linear system. It follows that the assumption of row incoherence (every two hyperplanes are not nearly parallel) is necessary to ensure that the bad iterate is unlikely to move the iterate too far from the true solution. The incoherence assumption was also introduced in \cite{haddock2022quantile,cheng2022block} to ensure that the quantile is representative.

\begin{algorithm}[H]
	\caption{quantile-based randomized sparse Kaczmarz (Quantile-RaSK)}\label{a1}
	\begin{algorithmic}[1]
		\State \textbf{Input}: Given $A \in \mathbb{R}^{m \times n}$, $x_{0}=x_{0}^*=0\in \mathbb{R}^{n}$,
		$b \in \mathbb{R}^{m},$ and parameters $\lambda, q,N$
		\State \textbf{Ouput}: solution of $\min_{x\in \mathbb{R}^n}\lambda\|x\|_{1}+\frac{1}{2}\|x\|_2^2~ \st~Ax=b$
		\State normalize $A$ by row
		\State \textbf{for} $k=0,1,2, \ldots$ do
		\State $\quad$$\quad$compute $N_1=\{1\leq i_k\leq m:|\langle x_k,a_{i_k}\rangle-b_{i_k}|\}$
		\State $\quad$$\quad$compute the $q$-quantile of $N_1$:  $Q_k=Q_q(x_k,1\leq i_k\leq m)$
		\State $\quad$$\quad$consider $N_2=\{1\leq i_k\leq m:|\langle x_k,a_{i_k}\rangle -b_{i_k}|\leq Q_k\}$
		\State $\quad$$\quad$sample ${i_k}$ at random with probability $ p_i=\frac{1}{qm}$, $i\in N_2$
		\State $\quad$$\quad$\textbf{switch} Type of step:
		\State $\quad$$\quad$\textbf{case1: inexact step}
		\State $\quad$$\quad$$\quad$$\quad$compute $t_k=\langle a_{i_k},x_k\rangle-b_{i_k}$
		\State $\quad$$\quad$\textbf{case2: exact step}
		\State $\quad$$\quad$$\quad$$\quad$compute $t_k=\arg\min_{t\in \mathbb{R}}f^*(x_k^*-t\cdot a_{i_k})+t\cdot b_{i_k}$
		\State $\quad$$\quad$\textbf{endSwitch}
		\State $\quad$$\quad$update $x_{k+1}^*=x_k^*-t_k\cdot a_{i_k}$
		\State $\quad$$\quad$update $x_{k+1}=\mathcal{S}_{\lambda}(x_{k+1}^*)$
		\State $\quad$$\quad$increment $k=k+1$
		\State \textbf{until} a stopping criterion is satisfied
	\end{algorithmic}
\end{algorithm}

\subsection{Quantile-RaSK for corrupted and noisy linear  system}
In this subsection, we consider a corrupted and noisy linear system
\begin{equation}
\label{3:linear}
Ax=b,
\end{equation}
where $b=\tilde{b}+b^c+r$ meaning that the observed data $b$ is the sum of the exact data $\tilde{b} = A\hat{x}$, the corruptions $b^c$ satisfying $\|b^c\|_0=\beta m$ and the noise $r\in \mathbb{R}^m$.
We want to recover the unknown solution $\hat{x}$ from $Ax=\tilde{b}$ by using Quantile-RaSK.
Before deducing the convergence result with respect to Bregman distance, we first bound the $q$-quantile of residuals, which is motivated by Lemma 1 in \cite{jarman2021quantilerk}.
\begin{lemma}
	\label{lemma1}
	Let $\beta < q < 1-\beta$ and apply Quantile-RaSK to problem (\ref{3:linear}) to generate iterate sequences $\{x_k\}$ and $\{Q_k\}$. Then
	\begin{eqnarray}
	Q_k\leq \frac{\sqrt{1-\beta}}{(1-\beta-q)\sqrt{m}}\sigma_{\max}\|x_k-\hat{x}\|+
	\frac{1-\beta}{1-\beta-q}\|r\|_{\infty}.\nonumber
	\end{eqnarray}
\end{lemma}

Now we turn to the first main result of this study.
\begin{theorem}[corrupted and noisy case]
	\label{key}
	Let the linear system (\ref{3:linear}) have a fraction $\beta$ of corrupted entries and the noise $r$. Assume that $\beta<q<1-\beta$ and the matrix $A\in\mathbb{R}^{m\times n}$ is full rank with unit-norm rows. If the following relationship holds	
	\begin{align}
	\label{condition2}
	&\frac{2\sqrt{\beta(1-\beta)}}{1-\beta-q}\left(
	\sigma_{\max }
	+\sqrt{\frac{(1-\beta)m}{n}}
	\right)
	+
	\frac{\beta(1-\beta)}{(1-\beta-q)^2}\left(\sigma_{\max }+2\sqrt{\frac{(1-\beta)m}{n}}
	\right)\nonumber\\
	&~~~~~~~~~~~~~~~~~~~~~~~~~~~~~~<
	\frac{\alpha(q-\beta)}{2q}\frac{\tilde{\sigma}_{q-\beta,\min}^2}{\sigma_{\max }},
	\end{align}
	then Quantile-RaSK converges linearly, in the sense that\\
	(a) the iterates $x_k$ generated by Quantile-RaSK with inexact step satisfy
	\begin{align}
	&\mathbb{E}[D_f^{x_{k+1}^*}(x_{k+1},\hat{x})]\leq
	\left(1-C_1\right)\mathbb{E}[D_f^{x_{k}^*}(x_{k},\hat{x})]
	+C_2\|r\|_{\infty}^2.\nonumber
	\end{align}
	(b) the iterates $x_k$ generated by Quantile-RaSK with exact step satisfy
	\begin{align}
	&~~~~\mathbb{E}[D_f^{x_{k+1}^*}(x_{k+1},\hat{x})]\nonumber\\
	&\leq
	\left(1-C_1\right)\mathbb{E}[D_f^{x_{k}^*}(x_{k},\hat{x})]
	+C_2\|r\|_{\infty}^2
	+\frac{2}{\sqrt{qm}}\|r\|_{\infty}\cdot\|A\|_{1,2}
	+\frac{2\lambda}{qm}\|b-\tilde{b}\|\cdot\|A\|_{1,2},\nonumber
	\end{align}
	where
	\begin{eqnarray}
	\label{constant:1}
	&C_1=\frac{\alpha(q-\beta)}{2q^2}\frac{\tilde{\sigma}_{q-\beta,\min}^2}{m}
	-
	\frac{2\sqrt{\beta(1-\beta)}}{q(1-\beta-q)}\left(
	\frac{\sigma_{\max }^2}{m}+\frac{\sqrt{1-\beta}\sigma_{\max }}{\sqrt{mn}}
	\right)
	-
	\frac{\beta(1-\beta)}{q(1-\beta-q)^2}\left(\frac{\sigma_{\max }^2}{m}+\frac{2\sqrt{1-\beta}\sigma_{\max}}{\sqrt{mn}}
	\right),\nonumber\\
	&C_2=\frac{\sqrt{\beta}(1-\beta)}{q(1-\beta-q)}\left(1+\frac{\sqrt{\beta(1-\beta)}}{1-\beta-q}\right)\sqrt{\frac{n}{m}}\sigma_{\max}
	+\frac{1}{2}\frac{\beta(1-\beta)^2}{q(1-\beta-q)^2}
	+\frac{1}{2}.\nonumber
	\end{eqnarray}
\end{theorem}

When the linear system (\ref{3:linear}) is only corrupted, i.e., $r=0$, Theorem \ref{key} turns into the following theorem. The proof is similar to Theorem \ref{key}, and we omit it.
\begin{theorem}[corrupted case]
	\label{Th:3.1}
	Let $r=0$ in linear system (\ref{3:linear}).
	Assume that $\beta< q < 1-\beta$ and the matrix $A\in\mathbb{R}^{m\times n}$ in linear system (\ref{3:linear}) is full rank with unit-norm rows. If the following relationship holds
	\begin{equation}
	\label{condition:2}
	\frac{2q}{q-\beta}\cdot\left(
	\frac{2\sqrt{\beta}}{\sqrt{1-q-\beta}}+\frac{\beta}{1-q-\beta}
	\right)
	\cdot\frac{|\hat{x}|_{\min}+2\lambda}{|\hat{x}|_{\min}}
	<
	\frac{\tilde{\sigma}_{q-\beta,\min}^2(A)}{\sigma_{\max}^2(A)},
	\end{equation}
	then Quantile-RaSK converges linearly for the linear system with any $\beta$-corruptions, in the sense that\\
	(a) the iterates $x_k$ generated by Quantile-RaSK with inexact step satisfy
	\begin{eqnarray}
	\label{th1}
	\mathbb{E}[D_f^{x_{k+1}^*}(x_{k+1},\hat{x})]
	\leq
	(1-C)\mathbb{E}[D_f^{x_{k}^*}(x_{k},\hat{x})];
	\end{eqnarray}	
	(b) the iterates $x_k$ generated by Quantile-RaSK with exact step satisfy
	\begin{equation}
	\label{th2}
	\mathbb{E}[D_f^{x_{k+1}^*}(x_{k+1},\hat{x})]
	\leq
	(1-C)\mathbb{E}[D_f^{x_{k}^*}(x_{k},\hat{x})]+\frac{2\lambda}{qm}\|b-\tilde{b}\|\cdot\|A\|_{1,2},\nonumber
	\end{equation}
	where
	$$C
	=\frac{q-\beta}{2q^2m}\cdot\frac{|\hat{x}|_{\min}}{|\hat{x}|_{\min}+2\lambda}\cdot\tilde{\sigma}_{q-\beta,\min}^2(A)
	-
	\frac{1}{qm}
	\left(
	\frac{2\sqrt{\beta}}{\sqrt{1-q-\beta}}+\frac{\beta}{1-q-\beta}
	\right)\cdot\sigma_{\max}^2(A)>0.$$
	
\end{theorem}

\begin{remark}
	The convergence results in Theorem \ref{key} and Theorem \ref{Th:3.1} strictly depend on the condition (\ref{condition2}) and (\ref{condition:2}), respectively. They hold for small enough $\beta$ since the left-hand side of (\ref{condition2}) and (\ref{condition:2}) tend to zero as $\beta$ tends to zero.
\end{remark}

\begin{remark}
	Clearly, parameters $q,\beta$ should satisfy $\beta< q< 1-\beta$. It is natural for us to explore the best relationship between $\beta$ and $q$ for a more precise convergence rate.
	Since it is difficult for us to provide theoretical guidance,
	we will give an empirical result in the numerical part.
\end{remark}


\subsection{Subgaussian case}

In this subsection, we specialize $A$ to be a class of random matrices satisfying the following Assumption 1 and Assumption 2; then we can provide the bound of singular values $\tilde{\sigma}_{q-\beta,\min},\sigma_{\max}$ with a high probability.\\
\textbf{Assumption 1.}
All the rows $a_i$ of the matrix $A$ have unit norms and are independent. Additionally, for all $i\in\{1,\cdots,m\}$, $\sqrt{n}a_i$ is mean zero isotropic and has uniformly bounded subgaussian norm $\|\sqrt{n}a_i\|_{\psi_2}\leq K$.\\
\textbf{Assumption 2.}
Each entry $a_{ij}$ if $A$ has probability density function $\phi_{ij}(t)\leq D\sqrt{n}$ for all $t\in\mathbb{R}$. (The quantity $D$ is a constant which we will use throughout when referring to this assumption.)

Given that subgaussian random vectors tend to be almost orthogonal, the assumption of incoherence naturally holds in the subgaussian case. Although sampling a corrupted row leads to a bad projection, the assumption incoherence ensures that it is not bad enough and the next good iterate will make up for it. On the other hand, incoherence keeps the smallest singular value of submatrices of $A$ away from zero, where the convergence rate of Quantile-RaSKA is well-defined in theory.

\begin{lemma}
	\label{lemma:2.5}
	Assume that $A\in\mathbb{R}^{m\times n}$ satisfies Assumption 1 and Assumption 2 with constants $D$ and $K$.
	Let the corrupted parameter $\beta$ and the acceptable parameter $q$ satisfy $\beta<q<1-\beta$ and let $A$ be tall enough in the sense that there exist positive constants $D_1$ such that
	\begin{equation}
	\label{condition1}
	\frac{m}{n}>\frac{\textsc{D}_1}{q-\beta}\log\frac{DK}{q-\beta}.
	\end{equation}
	Then, with probability at least $1-3\textrm{exp}(-d_1(q-\beta) m)$ the constant $\tilde{\sigma}_{q-\beta,\min}(A)$ can be lower bounded as follows
	\begin{eqnarray}
	\label{2:sigma}
	\tilde{\sigma}_{q-\beta,\min} \geq c_L (q-\beta)^{\frac{3}{2}}\sqrt{\frac{m}{n}}.
	\end{eqnarray}	
	Additionally, with a probability of at least $1-2\textrm{exp}(-d_2m)$ the largest singular value of $A$ can be  upper bounded as follows
	\begin{eqnarray}
	\label{2:sigmamax}
	\sigma_{\max}\leq c_K\sqrt{\frac{m}{n}}.
	\end{eqnarray}	
	Above, $D_1, d_1, d_2, c_L, c_K$ are positive constants.
\end{lemma}

The proof is very similar to that of Proposition 1 and Theorem 3 in \cite{haddock2022quantile}; we omit the detail. 
Equipped with these bounds, we proceed to bound $C_1,C_2$ in Theorem \ref{key}. Then with a high probability
$$C_1\geq \frac{c_1'}{n}, ~ C_2\leq c_2',$$
where $c_1'$ and $c_2'$ are positive constants only depending on $q$ and $\beta$. Therefore, we conclude the following lemma directly from Theorem \ref{key}.
\begin{lemma}
	Assume that the normalized random matrix $A$ satisfies Assumption 1 and Assumption 2. Let the corrupted parameter $\beta$ and the acceptable parameter $q$ satisfy $\beta<q<1-\beta$ and let $A$ be tall enough in (\ref{condition1}). For any $k\in\mathbb{N}$, with a high probability, the iterates $x_k$ generated by Quantile-RaSK with inexact step satisfy
	\begin{equation}
	\label{subgaussian:1}
	\mathbb{E}[D_f^{x_{k+1}^*}(x_{k+1},\hat{x})]\leq
	\left(1-\frac{c_1'}{n}\right)\mathbb{E}[D_f^{x_{k}^*}(x_{k},\hat{x})]
	+c_2'\|r\|_{\infty}^2,
	\end{equation}
	and
	the iterates $x_k$ generated by Quantile-RaSK with exact step satisfy
	\begin{align}
	\mathbb{E}[D_f^{x_{k+1}^*}(x_{k+1},\hat{x})]\leq
	\left(1-\frac{c_1'}{n}\right)\mathbb{E}[D_f^{x_{k}^*}(x_{k},\hat{x})]
	&+c_2'\|r\|_{\infty}^2+\frac{2}{\sqrt{qm}}\|r\|_{\infty}\cdot\|A\|_{1,2}\nonumber\\
	&+\frac{2\lambda}{qm}\|b-\tilde{b}\|\cdot\|A\|_{1,2}.
	\end{align}
\end{lemma}

\begin{remark}
	According to Lemma \ref{lemma1} in the subgaussian case, the distance between the sampled hyperplane and the current iterate can be bounded by $c/\sqrt{n}$. Therefore, even if we project onto a corrupted row, the iterate is not too far from the true solution. Intuitively, this guarantees the convergence of the Quantile-RaSK method.
\end{remark}

\section{The Quantile-RaSK method with block accelerations}
In order to speed up the Quantile-RaSK method in Section 3, a randomized sparse Kaczmarz algorithm based on quantile and block acceleration (Quantile-RaSKA) is proposed. Under the existence of $\beta m$ corruptions, we show that Quantile-RaSKA enjoys a speedup by a factor of $n$ over Quantile-RaSK in the subgaussian setting.

\subsection{The proposed  Quantile-RaSKA method}
We proceed to design a more general sparse solver for a corrupted and noisy linear system (\ref{3:linear}).
We continue to use the statistics quantile to detect corrupted equations. In addition, we use the averaged block randomized sparse Kaczmarz method to achieve distributed computing and acceleration.
Different from Algorithm 1 using an uncorrupted hyperplane, the averaged block method needs Bregman projections onto all uncorrupted rows, and then take average of these projections as a new projection, which improves the utilization of matrix information. We provide an accelerated version of Quantile-RaSK, which is abbreviated as Quantile-RaSKA. The pseudocode is as follows.
\begin{algorithm}
	\caption{Quantile-RaSK with block acceleration (Quantile-RaSKA)}\label{a2}
	\begin{algorithmic}[1]
		\State \textbf{Input}: Given $A \in \mathbb{R}^{m \times n}$, $x_{0}=x_{0}^*=0\in \mathbb{R}^{n}$,
		$b \in \mathbb{R}^{m},$ and parameters $\lambda, q,N$
		\State \textbf{Ouput}: solution of $\min_{x\in \mathbb{R}^n}\lambda\|x\|_{1}+\frac{1}{2}\|x\|_2^2~ \st~Ax=b$
		\State normalize $A$ by row
		\State \textbf{for} $k=0,1,2, \ldots$ do
		\State $\quad$$\quad$compute $N_1=\{1\leq i_k\leq m:|\langle x_k,a_{i_k}\rangle-b_{i_k}|\}$
		\State $\quad$$\quad$compute the $q$-quantile of $N_1$:  $Q_k=Q_q(x_k,1\leq i_k\leq m)$
		\State $\quad$$\quad$consider $T=\{i\in \{1,\cdots,m\} ~|~ |\langle a_i,x_k\rangle-b_i|<Q_k\},|T|=\eta$
		\State $\quad$$\quad$$x_{k+1}^*=x_k^*-\frac{1}{\eta}\sum_{i\in T} w_i(\langle a_i,x_k\rangle-b_i) a_{i}$
		\State $\quad$$\quad$$x_{k+1}=\mathcal{S}_{\lambda}(x_{k+1}^*)$
		\State $\quad$$\quad$increment $k=k+1$
		\State \textbf{until} a stopping criterion is satisfied
	\end{algorithmic}
\end{algorithm}

\begin{remark}
	The Quantile-RaSKA can recover existing quantile-based randomized Kaczmarz variants, such as Quantile-RK (let $\lambda=0,\eta=1$), Quantile-RKA (let $\lambda=0$) and Quantile-RaSK (let $\eta=1$).
\end{remark}
Note that the stopping rule here is to reach the maximum number of iterations $N$, although it is suitable to use other stopping rules.
When dealing with the ill-conditioned problem (i.e. $\sigma_{q-\beta,\min }$ and $\sigma_{\min }$ small, $\|A\|_F$ and $\sigma_{\max }$ large),
we emphasize that the limitation of stepsize $w_i$ in the interval $(0,2)$ may slow down the convergence of Algorithm 2.
Intuitively, the extrapolated stepsize $w_i>0,i\in T$ makes the convergence rate of Algorithm 2 more flexible.

The presence of corruption may damage the performance of the averaged block variant.
As a block containing a corrupted row may lead to the iteration away from the true solution, the assumption of row incoherence is also needed to make intersection subspaces close to individual projection.

\subsection{Quantile-RaSKA method for corrupted and noisy linear  system}
For the sake of simplification, we use a constant stepsize $w$. To investigate the linear convergence of Quantile-RaSKA, we start by introducing important preliminary results from \cite{tondji2023faster,schopfer2019linear,steinerberger2023quantile,zhang2022quantile}. The first is a generalization of the proximal point operator \cite{tondji2023faster}, which characterizes the error bound between two consecutive iterates. Its proof can refer to Lemma 4.3 in \cite{tondji2023faster}.

\begin{lemma}[\cite{tondji2023faster}]
	\label{lemma2}
	Let $f,\Phi:\mathbb{R}^n\rightarrow\mathbb{R}\cup \{+\infty\}$ be convex, where $\textrm{dom}(f)=\mathbb{R}^n$ and $\textrm{dom}(\Phi)\neq \emptyset$.
	Let $\mathscr{X}\subset \textrm{dom}(\Phi)$ be nonempty and convex, $x_k\in\mathbb{R}^n,x_k^*\in\partial f(x_k)$.
	Assume that
	\begin{equation}
	x_{k+1} \in \underset{x \in \mathscr{X}}{\operatorname{argmin}}\left\{\Phi(x)+D_f^{x_k^*}\left(x_k, x\right)\right\}.\nonumber
	\end{equation}
	Then there exist subgradient $x_{k+1}^* \in \partial f\left(x_{k+1}\right)$ such that it holds
	\begin{equation}
	\label{eq4:1}
	\Phi(y)+D_f^{x_k^*}\left(x_k, y\right) \geq \Phi\left(x_{k+1}\right)+D_f^{x_k^*}\left(x_k, x_{k+1}\right)+D_f^{x_{k+1}^*}\left(x_{k+1}, y\right)
	\end{equation}
	for any $y\in\mathscr{X}$.
\end{lemma}

The following lemma provides the error bound of Bregman distance in terms of residuals.
\begin{lemma}[\cite{schopfer2019linear}]
	\label{lemma3}
	Let $\tilde{\sigma}_{\min }$ and $|\hat{x}|_{\min }$ be as defined above. Then for all $x \in \mathbb{R}^n$ with $\partial f(x) \cap \mathcal{R}\left(A^T\right) \neq \emptyset$ and all $x^*=A^T y \in \partial f(x) \cap \mathcal{R}\left(A^T\right)$ we have
	\begin{equation}
	\label{eq4:2}
	D_f^{x^*}(x, \hat{x}) \leq \gamma \cdot\|A x-b\|_2^2 ,
	\end{equation}
	where
	\begin{equation}
	\gamma=\frac{1}{\tilde{\sigma}_{\min }^2} \cdot \frac{1}{\alpha}.\nonumber
	\end{equation}
\end{lemma}

Equipped with the above lemmas, we are ready to provide the second main result in this paper. The following result shows that our method is guaranteed to converge at least linearly as long as $q,\beta$ and singular values satisfy the corresponding constraint.

\begin{theorem}[corrupted and noisy case]
	\label{th:noisy-general}
	Considering that the corrupted and noisy linear system (\ref{3:linear}) with corrupted term $\|b^c\|_0=\beta m$ and noise $r$, assume that $\beta<q<1-\beta$. Let $A\in\mathbb{R}^{m\times n}$ be of full rank with unit-norm rows.
	If we use the constant stepsize, and then the iterates generated by Quantile-RaSKA satisfy
	\begin{equation}
	\begin{aligned}
	\label{resc:1}
	D_f^{x_{k+1}^*}\left(x_{k+1}, \hat{x}\right)=(1-c_1^*w+c_2^*w^2)D_f^{x_k^*}\left(x_k, \hat{x}\right)
	+(c_3^*w+c_4^*w^2)\|r\|_{\infty}^2,
	\end{aligned}
	\end{equation}
	where
	$$c_1^*=c_1\alpha\frac{\tilde{\sigma}_{\min }^2\sigma_{q-\beta,\min}^2}{ m\sigma_{\max}^2}-c_2\alpha\frac{\tilde{\sigma}_{\min }^2}{ m}-c_4\alpha\frac{\tilde{\sigma}_{\min }^2}{2\sqrt{mn}\sigma_{\max}},~
	c_2^*=c_3\alpha\frac{\tilde{\sigma}_{\min }^2\sigma_{\max}^2}{ m^2}+c_5\alpha\frac{\tilde{\sigma}_{\min }^2\sigma_{\max}}{2 m^{3/2}n^{1/2}},$$
	$$c_3^*=c_4\frac{\sqrt{n}\sigma_{\max}}{2\sqrt{m}},~c_4^*=c_5\frac{\sqrt{n}\sigma_{\max}^3}{2m^{3/2}}+c_6\frac{\sigma_{\max}^2}{m},$$ and $c_i,i=1,\cdots,6$ are positive constants only depend on $q,\beta$. To guarantee convergence, we require
	$$0<c_1^*w-c_2^*w^2<1,~ c_1^*>0.$$
	
\end{theorem}

When the linear system is only corrupted, we can obtain a similar convergence result.
\begin{theorem}[corrupted case]
	\label{th:corrupt2}
	Considering that the corrupted linear system (\ref{3:linear}) with corrupted term $\|b^c\|_0=\beta m$, assume that $\beta<q<1-\beta$. Let $A\in\mathbb{R}^{m\times n}$ be of full rank with unit-norm rows. If
	\begin{equation}
	\label{th1:16}
	0<\frac{w}{\gamma qm}\left( \frac{\sigma_{q-\beta,\min}^2}{\sigma_{\max}^2}
	-
	\frac{\sqrt{\beta}}{ \sqrt{1-q-\beta}}\right)
	-\frac{w^2\sigma_{\max}^2}{\gamma q^2 m^2}
	\left(
	1+\frac{\sqrt{\beta}}{\sqrt{1-q-\beta}}
	\right)^2<1,
	\end{equation}
	then the iterates of Quantile-RaSKA applied to the system satisfy
	\begin{equation}
	\begin{aligned}
	\label{th1:18}
	&~~~~D_f^{x_{k+1}^*}\left(x_{k+1}, \hat{x}\right)\\
	&\leq \left[1
	-\frac{w}{\gamma qm}\left( \frac{\sigma_{q-\beta,\min}^2}{\sigma_{\max}^2}
	-
	\frac{\sqrt{\beta}}{ \sqrt{1-q-\beta}}\right)
	+\frac{w^2\sigma_{\max}^2}{\gamma q^2 m^2}
	\left(
	1+\frac{\sqrt{\beta}}{\sqrt{1-q-\beta}}
	\right)^2
	\right]
	D_f^{x_k^*}\left(x_k, \hat{x}\right).
	\end{aligned}
	\end{equation}
\end{theorem}

\begin{remark}
	The convergence of Quantile-RaSKA in corrupted case heavily depends on the condition (\ref{th1:16}). Intuitively, it holds when $\beta$ is enough small and $m$ is enough large.
\end{remark}

Although there is no special limit on the measurement matrix in Theorem \ref{th:noisy-general} and Theorem \ref{th:corrupt2}, the convergence result
is complicated because it is dependent on some singular values. In the following part, the specific type of
the random measurement matrix is considered to simplify the convergence result.

\subsection{Subgaussian case}
Here we restrict $A$ to the class of random matrices satisfying Assumption 1 and Assumption 2.
We can simplify Theorem \ref{th:noisy-general} by using the bound of singular values $\sigma_{\max },\sigma_{q-\beta,\min },\tilde{\sigma}_{\min}$.
\begin{lemma}
	\label{coro4:1}
	Recall the definition of $\tilde{\sigma}_{q-\beta,\min }$ and Lemma \ref{lemma:2.5}, and we have
	$$	
	\sigma_{q-\beta,\min }
	\geq
	\tilde{\sigma}_{q-\beta,\min }
	\geq
	c_L (q-\beta)^{\frac{3}{2}}\sqrt{\frac{m}{n}},~
	\tilde{\sigma}_{\min}
	\geq\tilde{\sigma}_{q-\beta,\min }
	\geq
	c_L (q-\beta)^{\frac{3}{2}}\sqrt{\frac{m}{n}},
	$$
	with probability at least $1-3\exp(-c_1(q-\beta)m)$.
\end{lemma}

\begin{lemma}
	\label{lemma4:3}
	Assume that the random matrix $A$ satisfies Assumption 1 and Assumption 2. Let the corrupted parameter $\beta$ and the acceptable parameter $q$ satisfy $\beta<q<1-\beta$ and let $A$ be tall enough in (\ref{condition1}). For any $k\in\mathbb{N}$, the iterates $x_k$ generated by Quantile-RaSKA satisfy
	\begin{equation}
	\label{step:1}
	D_f^{x_{k+1}^*}\left(x_{k+1}, \hat{x}\right)\leq
	\left(
	1-\frac{a_1}{n}w+\frac{a_2}{n^2}w^2
	\right)D_f^{x_{k}^*}\left(x_{k}, \hat{x}\right)
	+\left(
	a_3w+\frac{a_4}{n}w^2
	\right)\|r\|_{\infty}^2,
	\end{equation}
	where $a_i,i=1,\cdots,4$ are positive constants only with respect to $q,\beta$.
\end{lemma}

We can directly obtain Lemma \ref{lemma4:3} from Theorem \ref{th:noisy-general} and Lemma \ref{coro4:1}, so we omit its proof here.
Viewing the convergence rate in (\ref{step:1}) as a quadratic function of stepsize $w$, the optimal stepsize is $w^*=\frac{a_1}{2a_2} n$ and the corresponding optimal convergence result is
\begin{equation}
\label{resc:2}
D_f^{x_{k+1}^*}\left(x_{k+1}, \hat{x}\right)
\leq
(1-\frac{a_1^2}{4a_2})D_f^{x_k^*}\left(x_k, \hat{x}\right)
+\left(\frac{a_1a_3}{2a_2}+\frac{a_1^2a_4}{4a_2^2}\right)n \|r\|_{\infty}^2.
\end{equation}
In the subgaussian case, the optimal constant stepsize is proportional to $n$, which is relatively different from the empirical stepsize of $(0,2)$. Therefore, a large stepsize can accelerate the convergence rate of our proposed method. What should be
emphasized is that the convergence rate of Quantile-RaSKA is independent
of the dimensions of the measurement matrix $m,n$. Compared with the convergence rate of Quantile-RaSK in (\ref{subgaussian:1}), we find that Quantile-RaSKA is faster than Quantile-RaSK by a factor of $n$. In conclusion, the averaged block technique can greatly improve the convergence rate of Quantile-RaSK in the subgaussian case. Furthermore, the residual term is proportional to $n$, leading
to a larger convergence horizon for inconsistent linear systems when using the optimal stepsize.

\section{Numerical experiments}
We divide our experiments into two sections to test the performance of Quantile-RaSK and Quantile-RaSKA. 
In experiments, we generate measurement matrices $A$ in two ways: Gaussian models by MATLAB function 'randn' and real-data models by ATRtools toolbox \cite{hansen2007regularization}. Then, we obtain $\tilde{b}= A\hat{x}$, where $\hat{x}$ is created by the MATLAB function 'sparserandn' with sparsity $s$.
Furthermore, the corrupted and noisy data $b=\tilde{b}+b^c+r$, where noise $r\in \mathbb{R}^m$ and corruptions $b^c\in \mathbb{R}^m$ are taken from a uniform distribution.
We first explore the effect of corruptions $b^c$, stepsize $w$, $\beta$ and $q$ on our proposed algorithms. Second, we compare our proposed algorithms with Quantile-RKA by carrying out two experiments on simulated data and real-world data, respectively.
We take the median error of 100 trials in every trial and denote
$$\text{error}(k)=\frac{\|x_k-\hat{x}\|}{\|\hat{x}\|}$$
as the relative residual at the $k$-th iterate. When the maximum number of iterations $N$ is reached, the algorithm stops.

All experiments are performed with MATLAB (version R2021b) on a personal computer with 2.80-GHZ CPU(Intel(R) Core(TM) i7-1165G7), 16-GB memory, and Windows operating system (Windows 10).

\subsection{The effect of parameters }
In this subsection, we use the Gaussian model, i.e., a linear system with entries sampled from i.i.d. $N(0,1)$.

\subsubsection{The effect of corruptions \texorpdfstring{$b^c$}{}}
We perform experiments on $10000\times 500$ standardized matrices. Let $q=0.7,\beta=0.2,s=40,\lambda=1,w=1.5n$, the corruptions $b^c$ taken from $U(-k,k)$ for a range of $k$, and the noise $r$ comes from $U(-0.02,0.02)$.
We compare the performance of Quantile-ERaSK and Quantile-RaSKA for solving corrupted and noisy linear systems, respectively.

\begin{figure}[H]
	\centering
	\subfigure[The effect of corruptions on Quantile-ERaSK]{
		\label{exp1:a}
		\includegraphics[width=0.48\linewidth,height=0.32\textwidth]{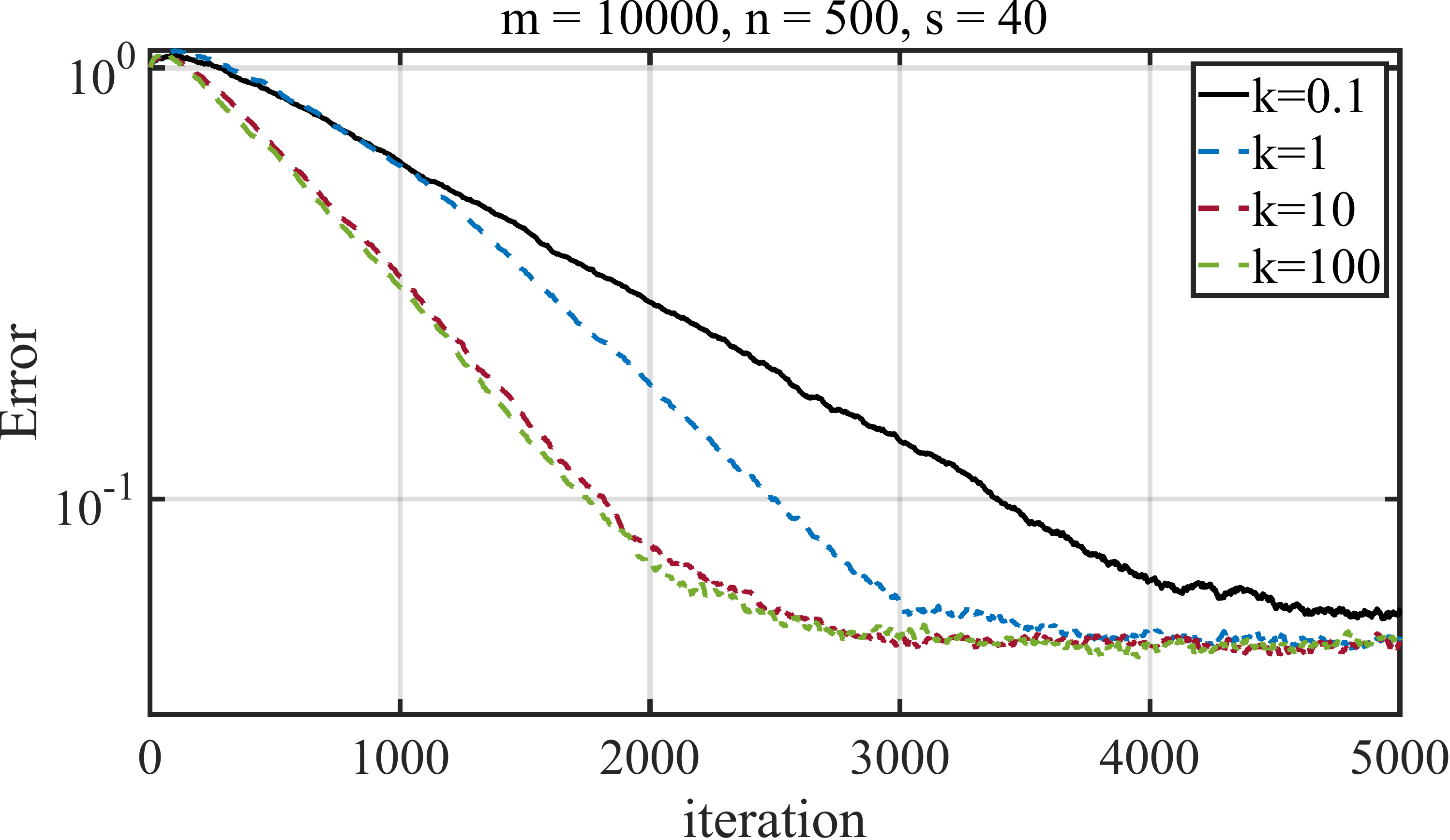}}
	\subfigure[The effect of corruptions on Quantile-RaSKA]{
		\label{exp1:b}
		\includegraphics[width=0.48\linewidth,height=0.32\textwidth]{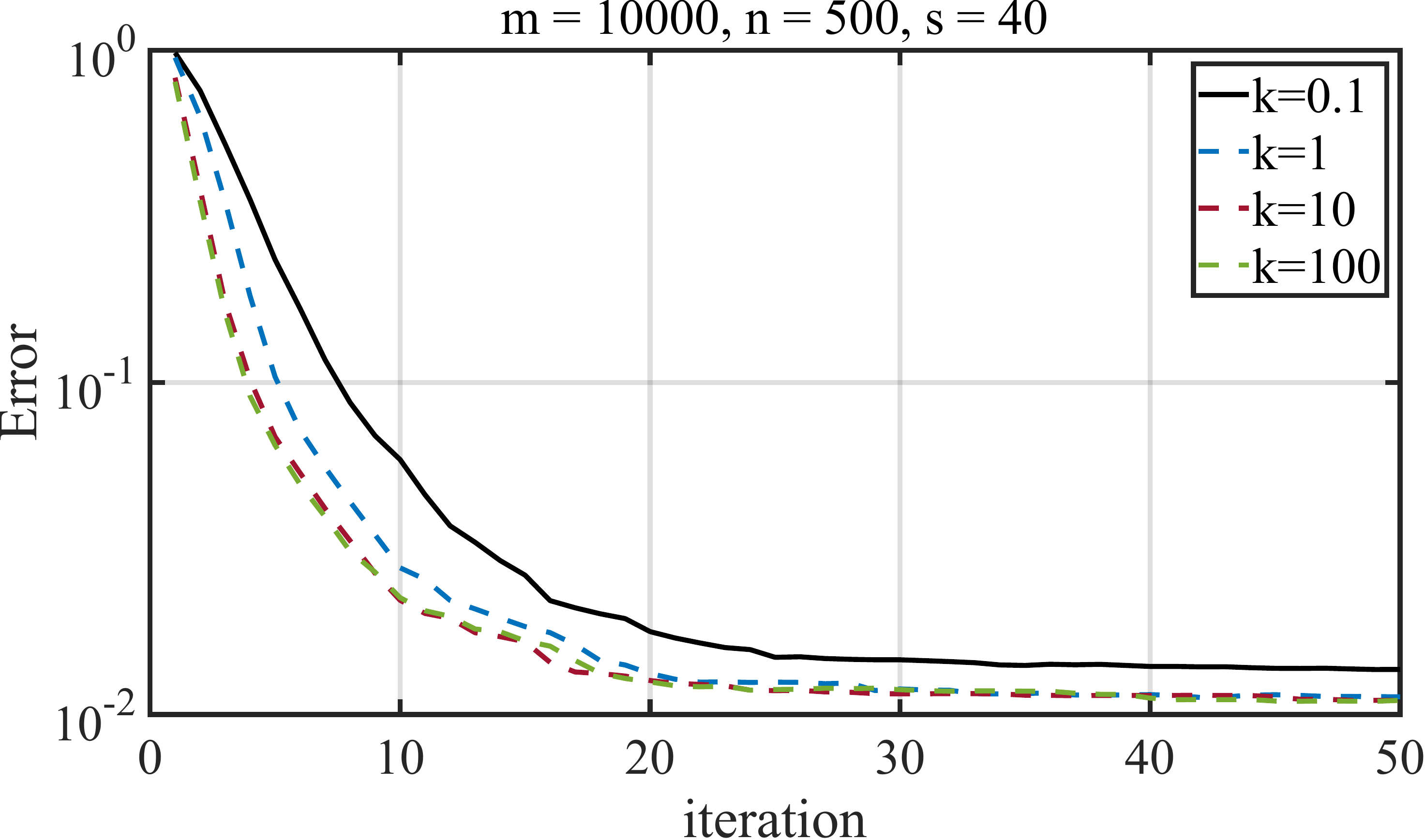}}
	\caption{The effect of the scale of corruptions on Quantile-ERaSK and Quantile-RaSKA}
	\label{fig:1}
\end{figure}

As shown in Fig. \ref{fig:1}, the Quantile-RaSKA method requires fewer iterations to achieve the same accuracy as the Quantile-ERaSK method. Apparently, the Quantile-RaSKA method exactly outperforms the Quantile-ERaSK method. Furthermore, the convergence rate of both Quantile-ERaSK and Quantile-RaSKA can be slightly improved as the
parameter $k$ increases. Informally, when corruptions are large relative to the noise, corruptions are easier to detect and hence our methods converge faster.

\subsubsection{The effect of stepsize \texorpdfstring{$w$}{}}
We use a constant stepsize for the linear systems with $A\in\mathbb{R}^{m\times n}$, where $m=10000,n=100:100:400$. Let $\beta=0.2,q=0.7,s=10,\lambda=1$ and record the relative residual at 20-th iterate in noiseless case and 1000-th iterate in noisy case. The corruptions $b^c$ comes from uniform distribution $U(-100,100)$ and the noise $r$ comes from $U(-0.02,0.02)$ in Fig. \ref{exp2_noise}.

According to Fig. \ref{fig:2}, it is apparent that the optimal stepsize is proportional to $n$ for the Gaussian models. In the noiseless and corrupted case, the optimal stepsize basically is around 1.6$n$ to 1.8$n$ for each $n$, while the optimal stepsize is around 0.3$n$ to 3.3$n$ in the noisy and corrupted case. The phenomenon verifies Lemma \ref{lemma4:3}, i.e., the optimal stepsize appears to scale with the number of columns $n$.

\begin{figure}[H]
	\centering
	\subfigure[The errors using Quantile-RaSKA in noiseless case]{
		\label{exp2_noiseless}
		\includegraphics[width=0.48\linewidth,height=0.32\textwidth]{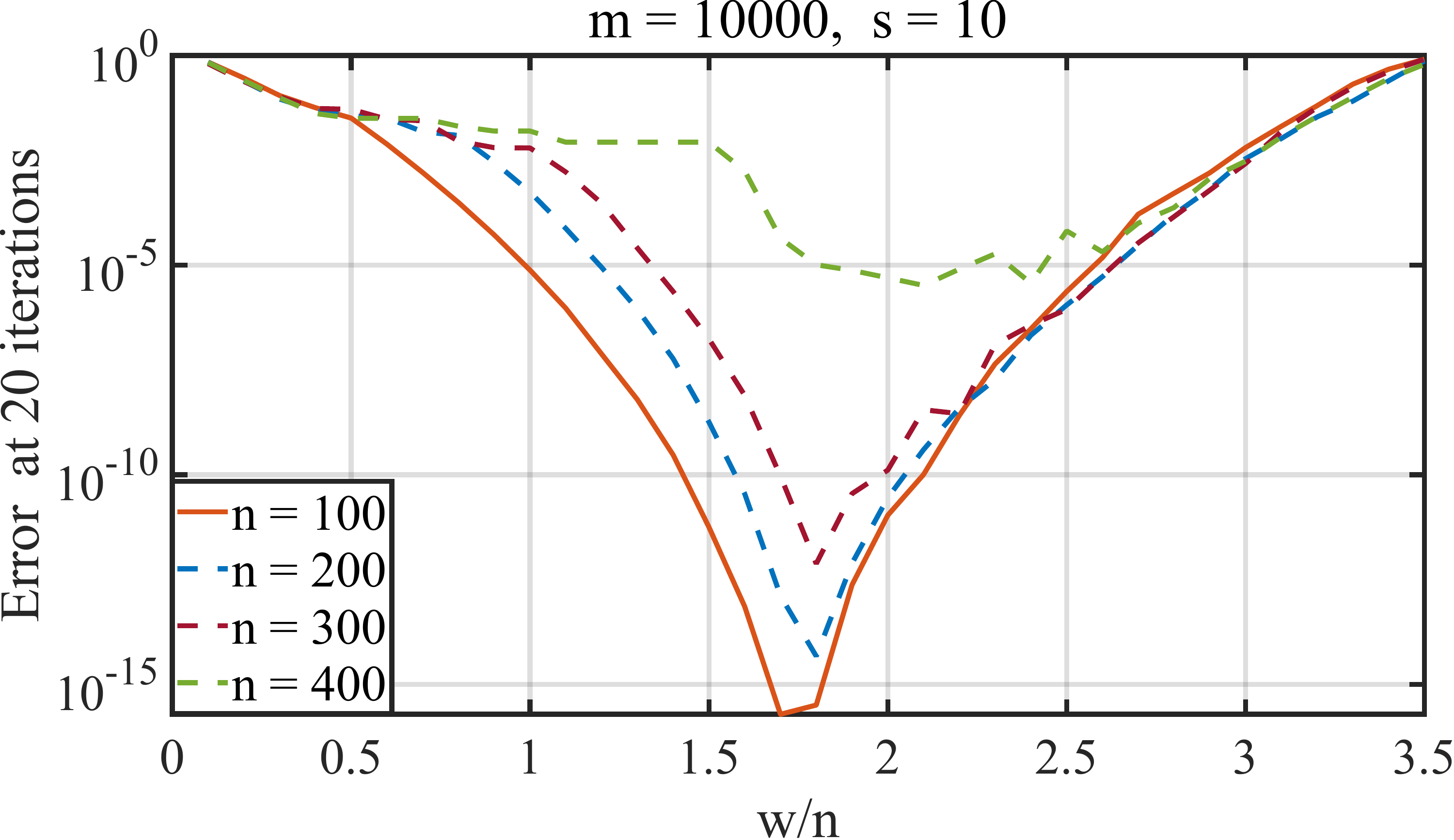}}
	\subfigure[The errors using Quantile-RaSKA in noisy case]{
		\label{exp2_noise}
		\includegraphics[width=0.48\linewidth,height=0.32\textwidth]{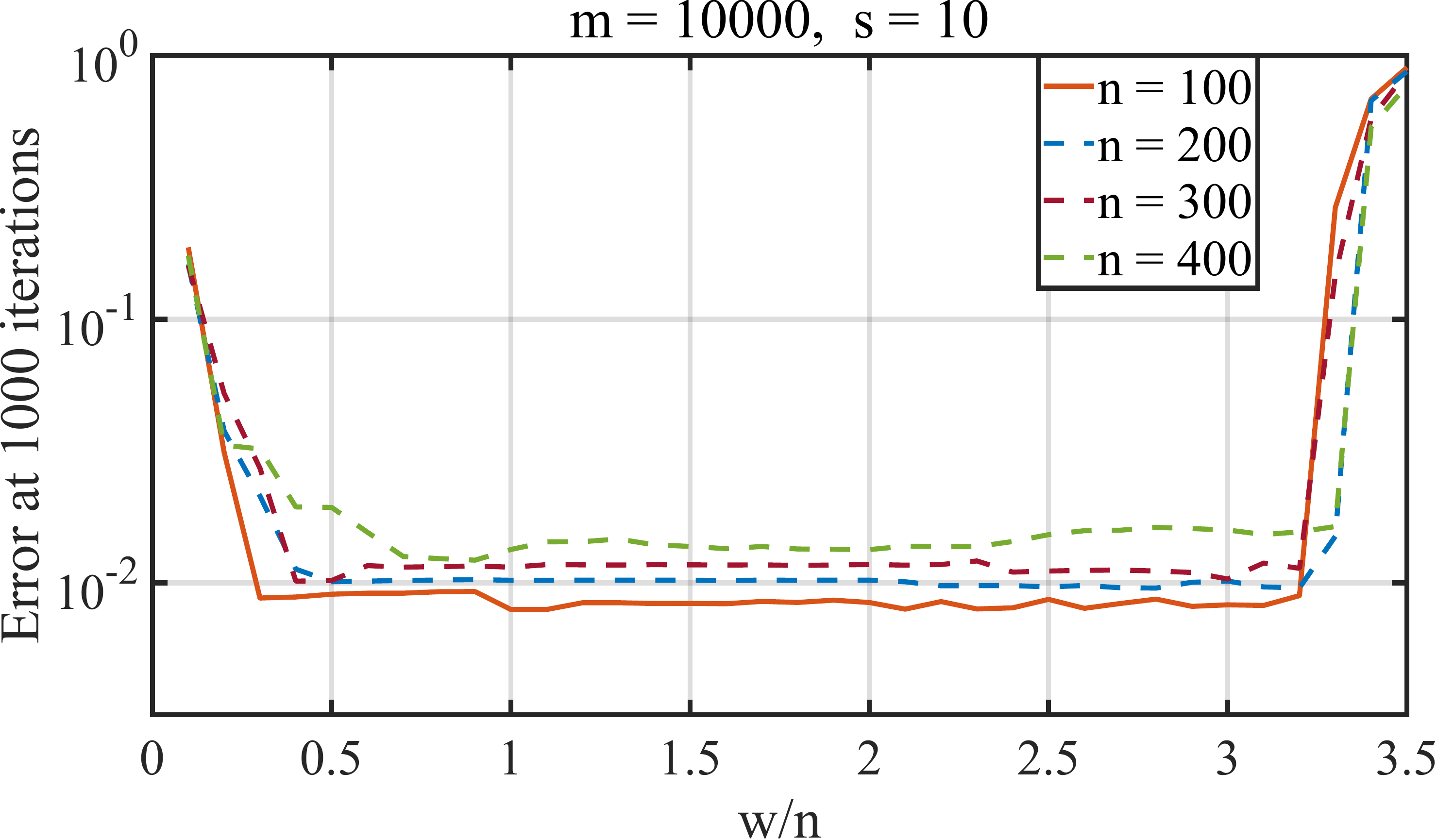}}
	\caption{The optimal stepsize of Quantile-RaSKA varying different $n$ for fixed $m=10000$}
	\label{fig:2}
\end{figure}

\subsubsection{The effect of \texorpdfstring{$\beta$}{} and \texorpdfstring{$q$}{}}
We have $\beta<q<1-\beta$ from the above theoretical analysis. Here, we carry out experiments to explore the best relationship between $q$ and $\beta$. Let $A\in\mathbb{R}^{10000\times 200},s=10,\lambda=1,w=1.7n,\beta=0.1:0.1:0.5$ and $q=0.1:0.1:1$. The corruption $b^c$ comes from $U(-100,100)$ and the noise comes from $U(-0.02,0.02)$.
We respectively study optimal parameters for Quantile-ERaSK and Quantile-RaSKA in both noiseless and noisy cases. To ensure good recovery accuracy, we record the error at 2000 iterations for Quantile-ERaSK and the error at 40 iterations for Quantile-RaSKA.

\begin{figure}[H]
	\centering
	\subfigure[The errors using Quantile-ERaSK in noiseless case]{
		\includegraphics[width=0.48\linewidth,height=0.32\textwidth]{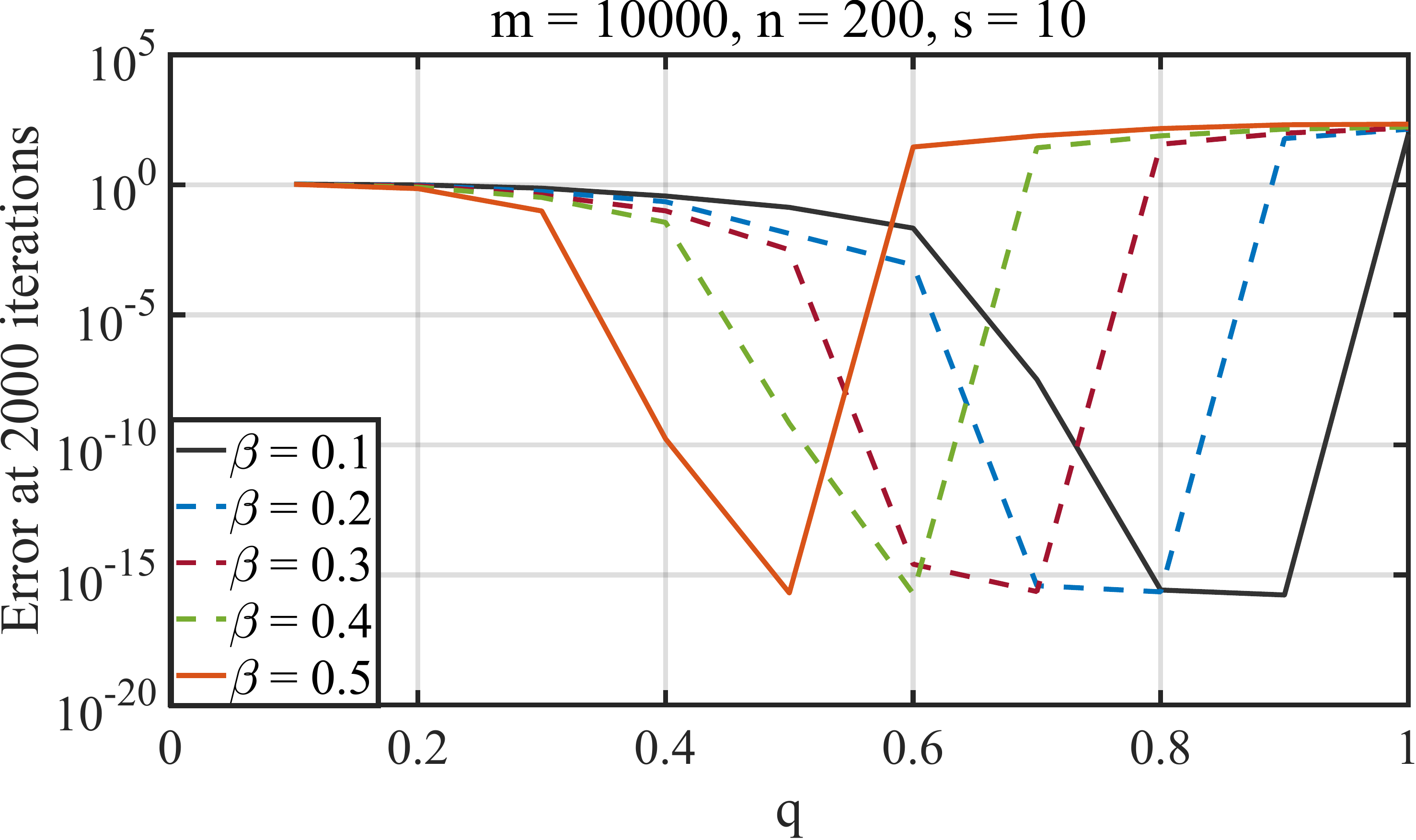}}
	\subfigure[The errors using Quantile-ERaSK in noisy case]{
		\includegraphics[width=0.48\linewidth,height=0.32\textwidth]{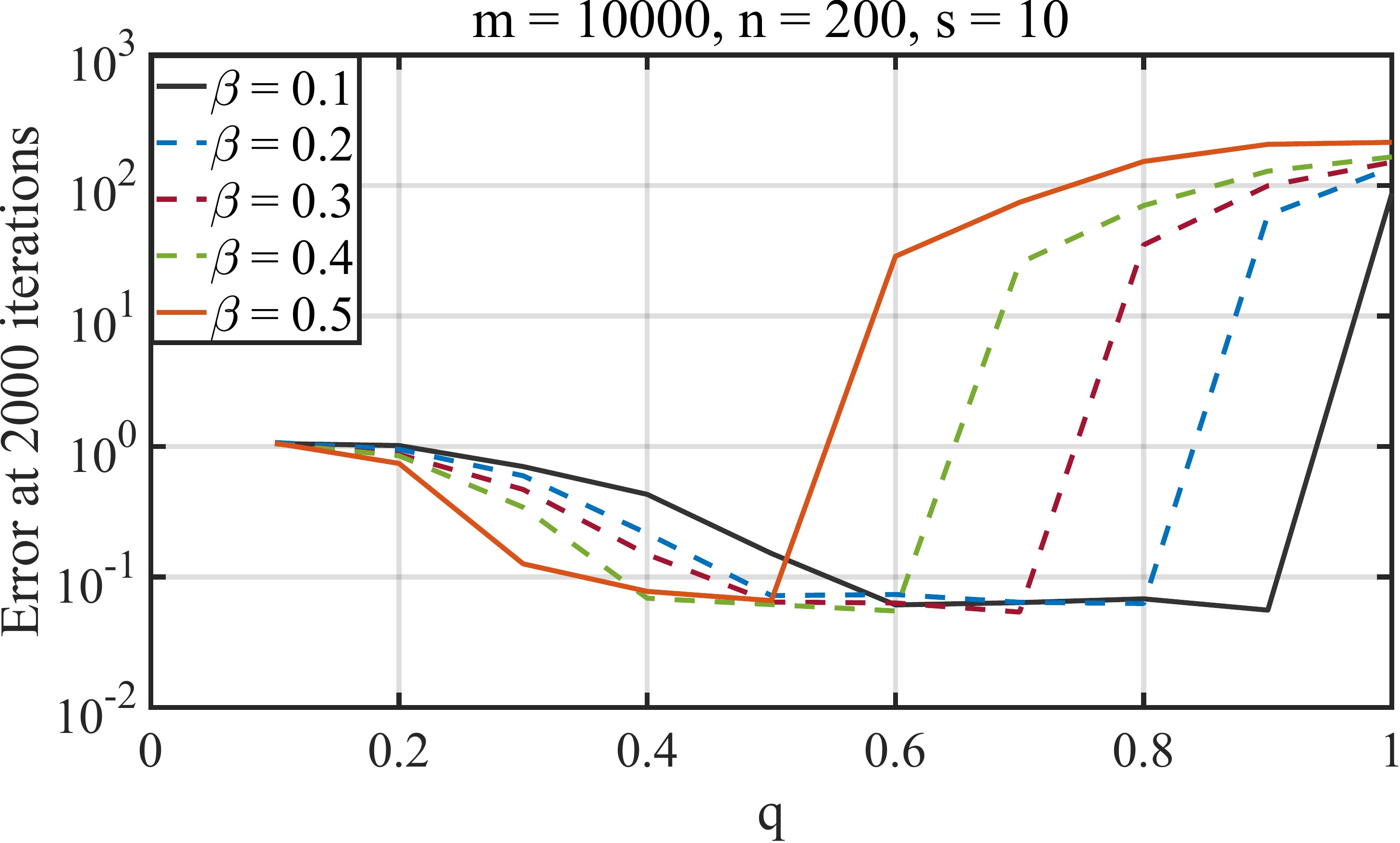}}
	\caption{The optimal $q$ of Quantile-ERaSK varying different corruption rates $\beta$}
	\label{fig:3}
\end{figure}

\begin{figure}[H]
	\centering
	\subfigure[The errors using Quantile-RaSKA in noiseless case]{
		\includegraphics[width=0.48\linewidth,height=0.32\textwidth]{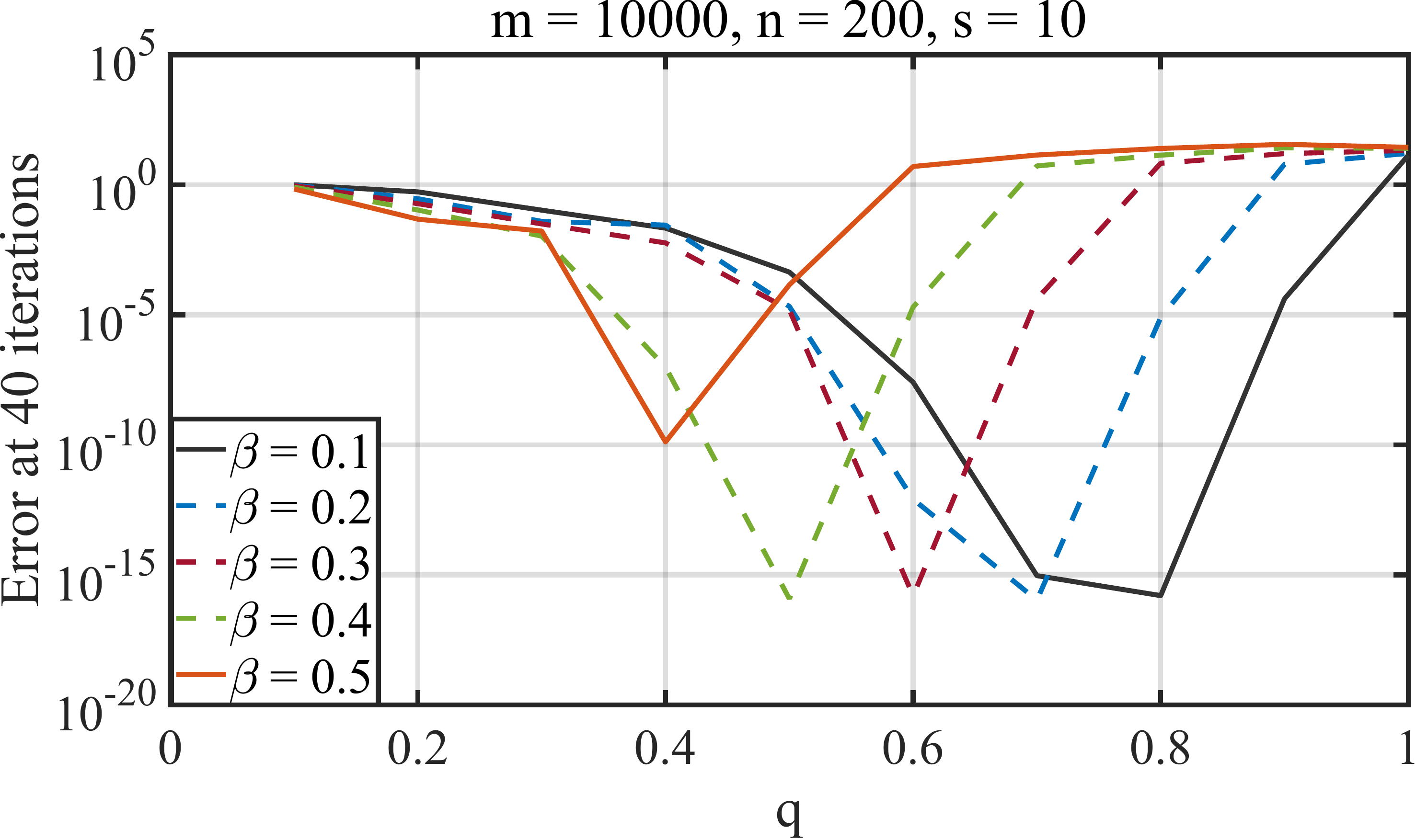}}
	\subfigure[The errors using Quantile-RaSKA in noisy case]{
		\includegraphics[width=0.48\linewidth,height=0.32\textwidth]{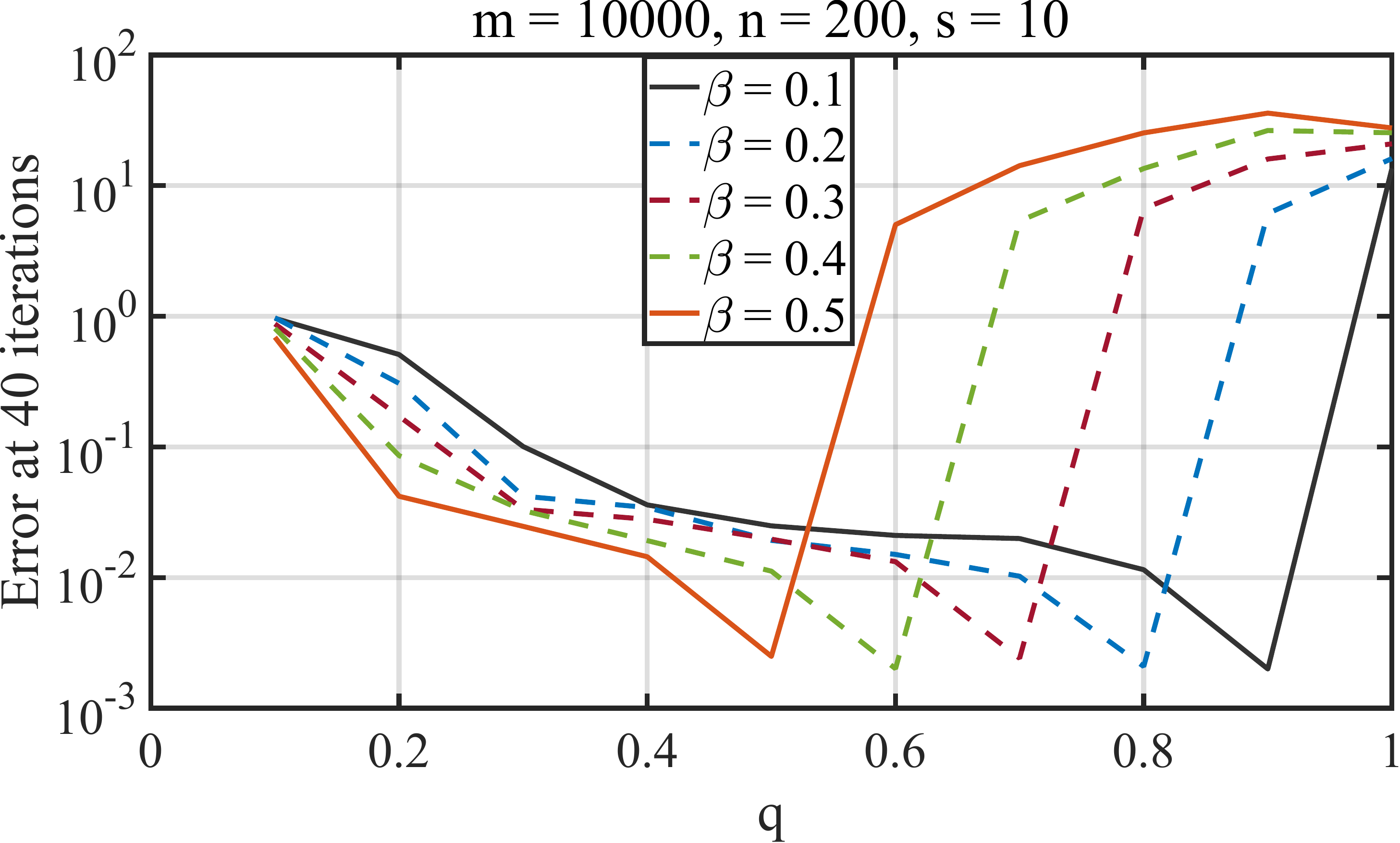}}
	\caption{The optimal $q$ of Quantile-RaSKA varying different corruption rates $\beta$}
	\label{fig:4}
\end{figure}

As indicated in Fig. \ref{fig:3} and Fig. \ref{fig:4}, the best $q$ takes $1-\beta$ for fixed $\beta$. Zooming in on the choice of $\beta$ and $q$, we find that the relative errors rapidly decrease when $q<1-\beta$ and continually increase when $q>1-\beta$. Therefore, we suggest taking $q$ slightly smaller than the suggested value $1-\beta$ in practical applications.

\subsection{Comparisions with different Quantile-based Kaczmarz variants}
In this subsection, we use two different kinds of models to compare the convergence rate and time complexity between Quantile-RKA, Quantile-RaSK and Quantile-RaSKA.

\subsubsection{The simulated model}

We construct the gaussian model with $A\in\mathbb{R}^{2000\times 200}$ and $s=10,\lambda=1,\beta=0.7,q=0.2$. The corruptions $b^c$ generates from $U(-100,100)$ and let the noise $r=0$. To ensure faster convergence, we take the step of Quantile-RaSKA and Quantile-RKA to be $w=1.7n$. The maxiter iteration is seted as 3000. 
According to Fig. \ref{fig:5}, we have that Quantile-RaSKA needs extremely few iterations and time compared to Quantile-RaSK with exact step (Quantile-ERaSK) and Quantile-RKA. Thus, Quantile-RaSKA outperforms other existing algorithms in terms of convergence rate and time complexity.

\begin{figure}[H]
	\centering
	\subfigure[The errors using Quantile-RaSKA in noiseless case]{
		\includegraphics[width=0.48\linewidth,height=0.32\textwidth]{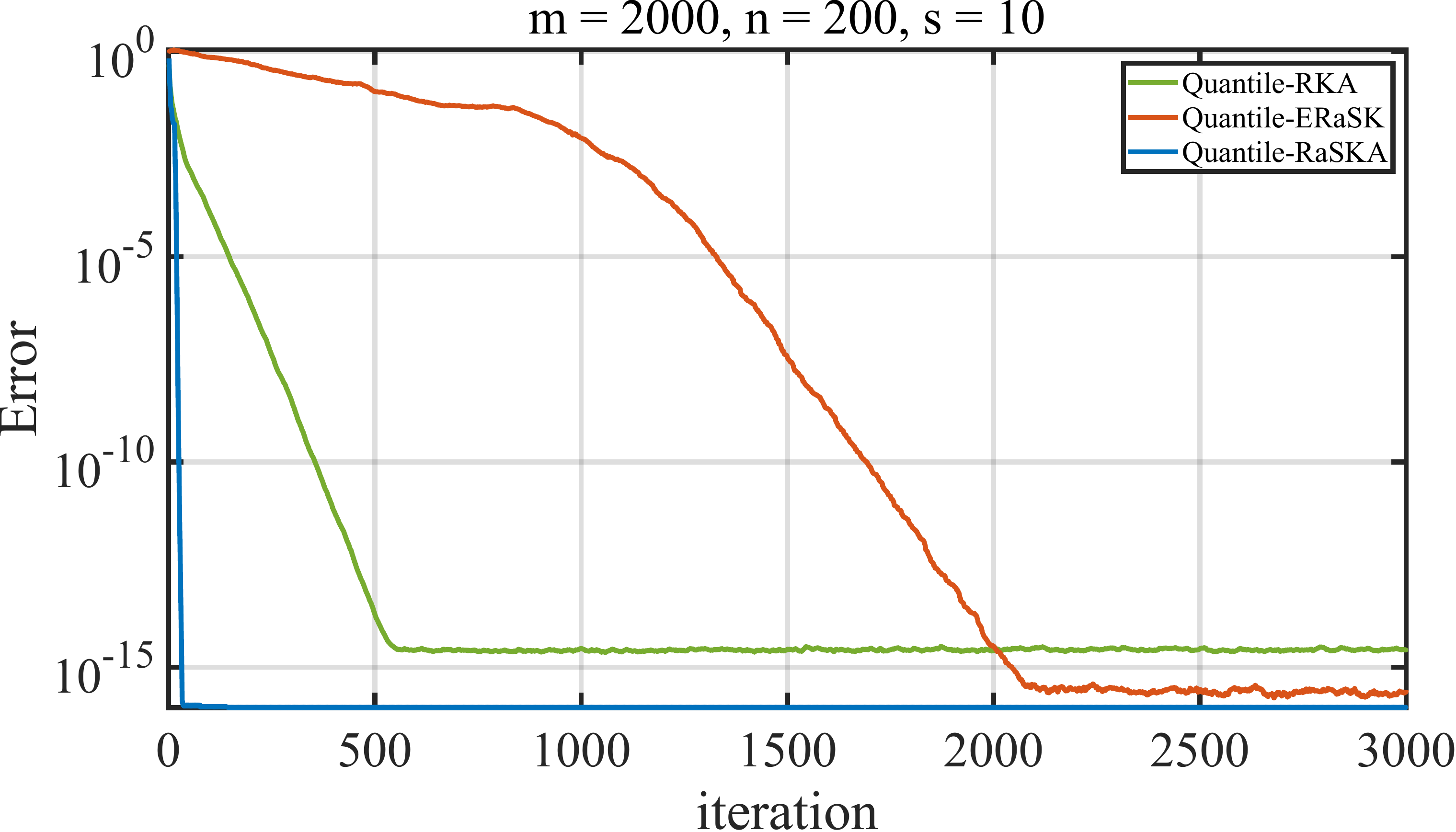}}
	\subfigure[The time using Quantile-RaSKA in noiseless case]{
		\includegraphics[width=0.48\linewidth,height=0.32\textwidth]{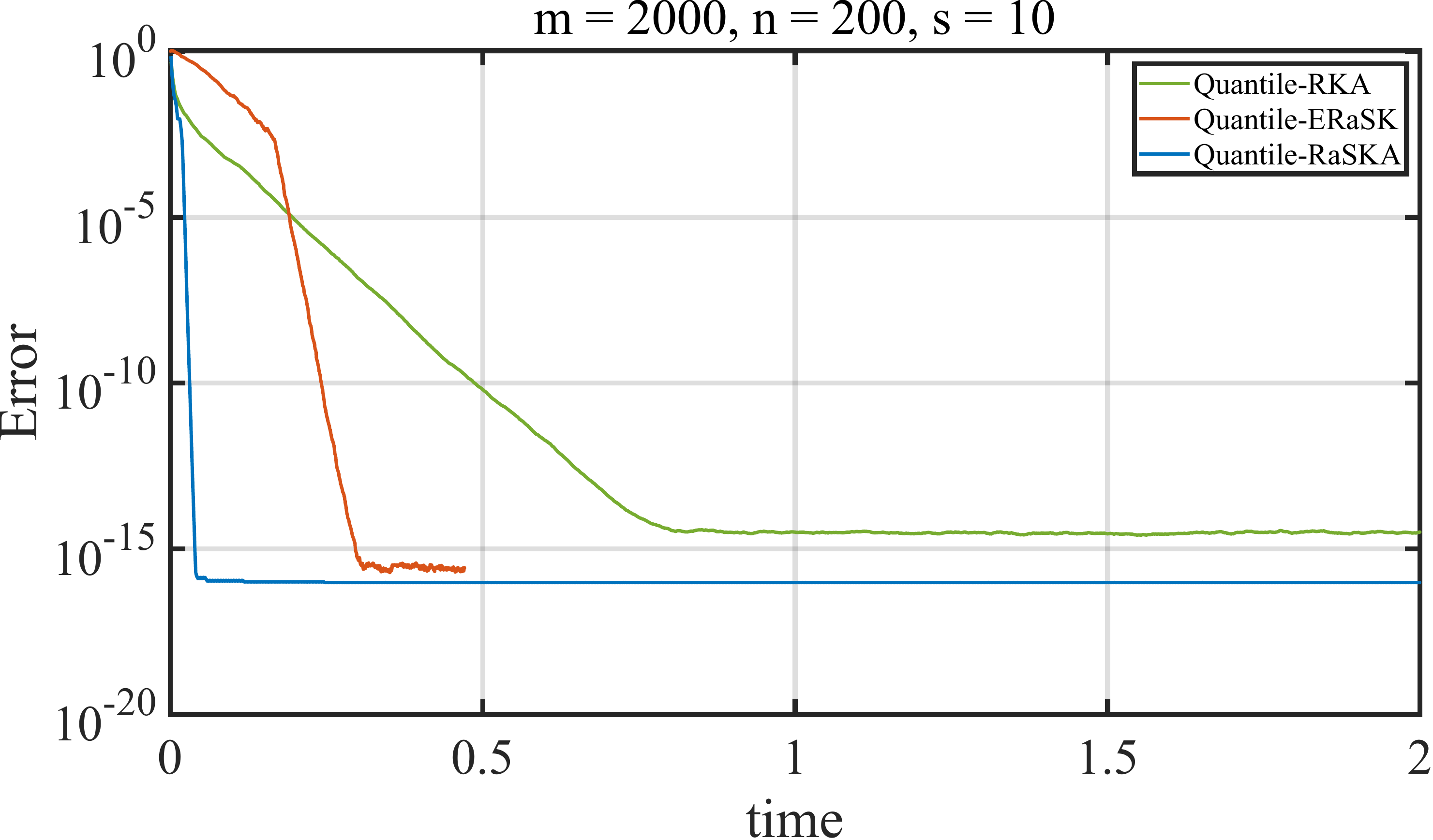}}
	\caption{The performance of different methods on simulated data}
	\label{fig:5}
\end{figure}

\subsubsection{The real-world data}

In this trial, we turn to study an academic tomography problem. The matrix $A\in\mathbb{R}^{1378\times 900}$ and true solution $\hat{x}\in\mathbb{R}^{900}$ are generated by the AIRtools toolbox (\url{http://www.imm.dtu.dk/~pcha/Regutools/}) \cite{hansen2007regularization}. The ground image is shown in Fig. \ref{ground}, which is sparse. Apparently, we have $\tilde{b}=A\hat{x}$ and we corrupt it by adding corruptions $b^c\sim U(-100,100)$ with $\beta=0.2$ and noise $r\sim U(-0.02,0.02)$. Let $q=0.7$.

We respectively use Quantile-RKA, Quantile-ERaSK, Quantile-RaSKA and to recover the ground image from the corrupted and noisy linear system.
From Fig. \ref{fig:6}, we can find that Quantile-RaSKA has advantages over the other methods in terms of the quality of the recovered image.
Although the part of the image information is damaged by large corruption and all the information is damaged by small noise, our proposed Quantile-RaSKA method can still recover the original image relatively well.
\begin{figure}[H]
	\centering
	\subfigure[Ground truth]{
		\label{ground}
		\includegraphics[width=0.16\linewidth,height=0.22\textwidth]{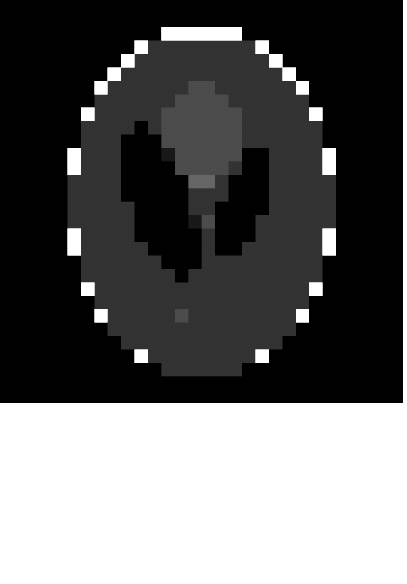}}
	\subfigure[Quantile-RKA]{
		\includegraphics[width=0.19\linewidth,height=0.25\textwidth]{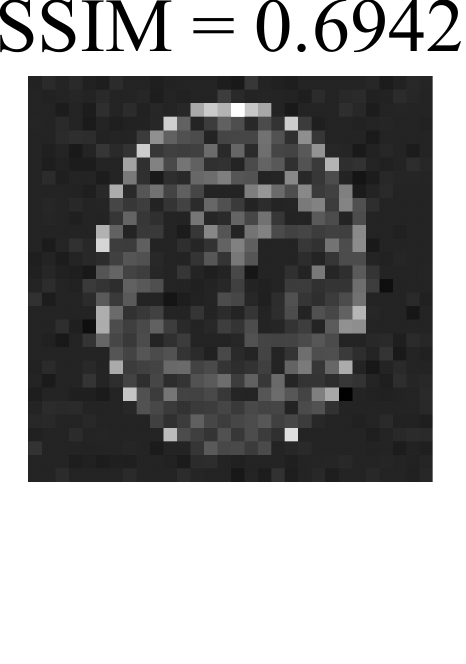}}
	\subfigure[Quantile-ERaSK]{
		\includegraphics[width=0.19\linewidth,height=0.25\textwidth]{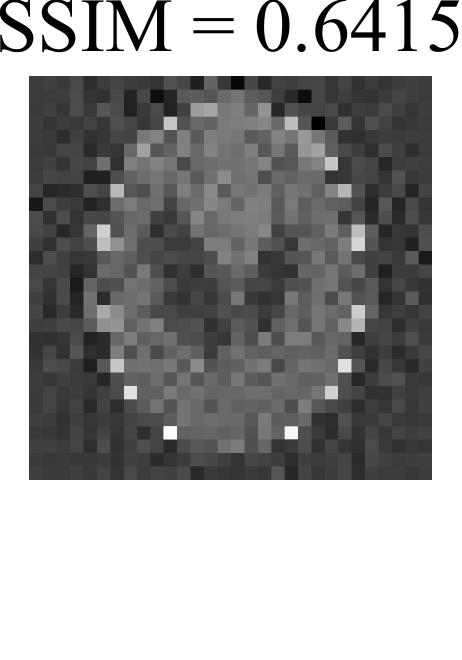}}
	\subfigure[QuantileRaSKA1]{
		\includegraphics[width=0.19\linewidth,height=0.25\textwidth]{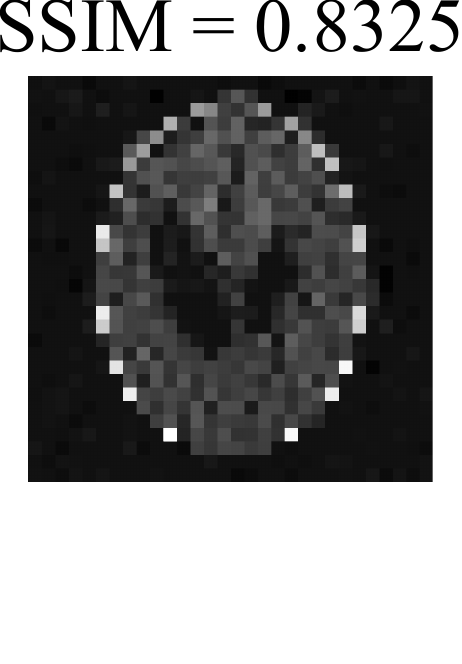}}
	\subfigure[QuantileRaSKA2]{
		\includegraphics[width=0.19\linewidth,height=0.25\textwidth]{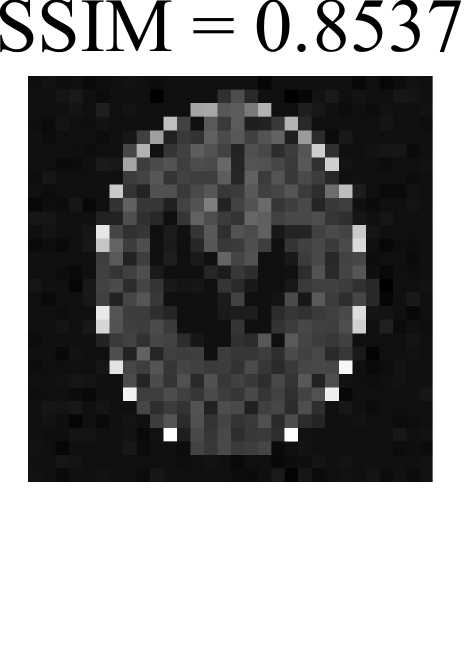}}
	\caption{The performance of different methods on real-world data: (a) Ground truth; (b) recovered by Quantile-RKA with $w=1$; (c) recovered by Quantile-ERaSK; (d) recovered by Quantile-RaSKA with $w=1$; (e) recovered by Quantile-RaSKA with $w=2$.}
	\label{fig:6}
\end{figure}

\section{Conclusion}
In order to find sparse solutions to large-scale corrupted and noise linear systems, two new Kaczmarz-type methods have been proposed in this paper, along with detailed convergence guarantees. The proposed methods not only unify many existing Kaczmarz-related algorithms, but also contribute as a more powerful algorithmic tool to solve more complicated linear systems that the current methods are unable to deal with. In the future, we would like to take the well-known heavy ball momentum technique \cite{polyak1964some} or adaptive step into our framework for possible speedup.

\bmhead{Acknowledgments}

This work was supported by the National Natural Science Foundation of China (No.11971480, No.61977065), the Natural Science Fund of Hunan for Excellent Youth (No.2020JJ3038), and the Fund for NUDT Young Innovator Awards (No.20190105).




\section*{Appendix}
\begin{appendices}
	\begin{proof}[Proof of Lemma \ref{lemma1}]
		Denote the set of all the indices of corrupted equations as $C$. The set excluding the $\beta m$ corrupted rows is denoted as $E$. For all $i\in E$, we have $b_i=\tilde{b}_i+r_i$ and $\langle a_i,\hat{x}\rangle=\tilde{b}_i=b_i-r_i$. Hence,
		\begin{align}
		\label{3:27}
		|\langle a_i,x_k-\hat{x}\rangle|
		=|\langle a_i,x_k\rangle-b_i+r_i|\nonumber
		\geq |\langle a_i,x_k\rangle-b_i|-\|r\|_{\infty}.
		\end{align}
		Thus,
		\begin{equation}
		\label{3:28}
		|\langle a_i,x_k\rangle-b_i|\leq
		|\langle a_i,x_k-\hat{x}\rangle|+\|r\|_{\infty}.
		\end{equation}
		Recall that $Q_k=Q_q(x_{k},1\leq i\leq m )$ and $|E|=(1-\beta)m$, meaning that at least $(1-\beta-q)m$ numbers of residuals $|\langle a_i,x_k\rangle-b_i|_{i=1}^m$ are at least $Q_k$. Then,
		\begin{align}
		\label{3:29}
		(1-\beta-q)mQ_k
		&\leq \sum_{i\in E}|\langle a_i,x_k\rangle-b_i|\nonumber\\
		&\leq \sum_{i\in E}(|\langle a_i,x_k-\hat{x}\rangle|+\|r\|_{\infty})\nonumber\\
		&\leq
		\left(\sum_{i\in E}|\langle a_i,x_k-\hat{x}\rangle|^2\right)^{\frac{1}{2}}\sqrt{|E|}+|E|\cdot\|r\|_{\infty}\nonumber\\
		&=
		\|A_{ E}(x_k-\hat{x})\|\sqrt{(1-\beta)m}
		+(1-\beta)m\|r\|_{\infty}\nonumber\\
		&\leq \sigma_{\max}\|x_k-\hat{x}\|\sqrt{(1-\beta)m}+
		(1-\beta)m\|r\|_{\infty}.
		\end{align}
		Therefore,
		\begin{align}
		\label{3:30}
		Q_k\leq \frac{\sqrt{1-\beta}}{(1-\beta-q)\sqrt{m}}\sigma_{\max}\|x_k-\hat{x}\|+
		\frac{1-\beta}{1-\beta-q}\|r\|_{\infty},
		\end{align}
		which completes the proof.
	\end{proof}
	
	\begin{proof}[Proof of Theorem \ref{key}]
		First of all, we give a proof sketch for Theorem \ref{key}.
		Denote the set consisting of all acceptable indices in the $k$-th iterate as
		\begin{equation}
		B=\{1\leq i\leq m:|\langle x_k,a_i\rangle|
		\leq Q_k\}.\nonumber
		\end{equation}
		The subset consisting of corrupted indices in $B$ is denoted as $S$; then the other indices in $B$ are in subset $B\backslash S$.
		Note that the used index in each iterate is sampled from the acceptable set $B$ in each iterate. This observation inspires us to consider the following splitting:
		\begin{align}
		\label{3:2}
		\mathbb{E}(D_f^{x_{k+1}^*}(x_{k+1},\hat{x}))=&
		P(i\in S)\mathbb{E}\left(D_f^{x_{k+1}^*}(x_{k+1},\hat{x})|i\in S\right)\nonumber\\
		&+
		P(i\in B\backslash S)\mathbb{E}\left(D_f^{x_{k+1}^*}(x_{k+1},\hat{x})|i\in B\backslash S\right).
		\end{align}
		Our remainder assignment is to estimate the split terms above.
		For clarity, we respectively take both inexact and exact steps into account, which are all divided into three steps.
		
		(a) In the inexact-step case,
		first we consider the uncorrupted equations indexed $i\in B\backslash S$. Since $b_{B\backslash S}^c=0$, the equations indexed $i\in B\backslash S$ satisfy
		\begin{eqnarray}
		A_{B\backslash S}x=\tilde{b}_{B\backslash S}=b_{B\backslash S}-r_{B\backslash S}.\nonumber
		\end{eqnarray}
		According to the convergence rate of RaSK in noisy case (\ref{rasknoisy:1}), we have
		\begin{align}
		\label{3:32}
		&~~~~~\mathbb{E}(D_f^{x_{k+1}^*}(x_{k+1},\hat{x})|i\in B\backslash S)\nonumber\\
		&\leq \left(1-\frac{1}{2}\cdot\frac{1}{\tilde{\kappa}(A_{B\backslash S})^2}\cdot\frac{|\hat{x}|_{\min}}{|\hat{x}|_{\min}+2\lambda}\right)
		D_f^{x_{k}^*}(x_{k},\hat{x})+\frac{1}{2}\frac{\|r_{B\backslash S}\|^2}{\|A_{B\backslash S}\|_F^2}\nonumber\\
		&\leq
		\left(1-\frac{1}{2}\cdot\frac{\tilde{\sigma}_{q-\beta,\min}^2}{qm}\cdot\frac{|\hat{x}|_{\min}}{|\hat{x}|_{\min}+2\lambda}\right)
		D_f^{x_{k}^*}(x_{k},\hat{x})
		+\frac{1}{2}\frac{qm}{|B\backslash S|}\|r\|_{\infty}^2,
		\end{align}
		where the last inequality follows from $\|r_{B\backslash S}\|^2\leq qm\cdot\|r\|_{\infty}^2$, $\|A_{B\backslash S}\|_F^2=|B\backslash S|$.
		
		Second, we consider the conditional expectation in $S$. In this case,
		\begin{align}
		A_Sx=\tilde{b}_S, b_S=\tilde{b}_S+b_S^c+r_S.\nonumber
		\end{align}
		Denote the orthogonal projection of true solution $\hat{x}\in H(a_{i_k},\tilde{b}_{i_k})$
		onto corrupted hyperplane $H(a_{i_k},b_{i_k})$ as $x_k^c$ \cite{needell2010randomized}, implying that
		\begin{equation}
		\label{3:5}
		x_k^c:=\hat{x}+(b_{i_k}-\tilde{b}_{i_k})a_{i_k}\in
		H(a_{i_k},b_{i_k}).
		\end{equation}
		Note that $x_{k+1}$ is a projection onto the hyperplane $H(a_{i_k},b_{i_k})$. According to Lemma \ref{lemma2.3}, it follows that
		\begin{equation}
		\label{3:6}
		D_f^{x_{k+1}^*}(x_{k+1},x_k^c)
		\leq
		D_f^{x_{k}^*}(x_{k},x_k^c)
		-
		\frac{1}{2}(\langle a_{i_k},x_k\rangle-b_{i_k})^2.
		\end{equation}
		By reformulating it, we obtain that
		\begin{equation}
		\label{3:7}
		D_f^{x_{k+1}^*}(x_{k+1},\hat{x})
		\leq
		D_f^{x_{k}^*}(x_{k},\hat{x})
		-
		\frac{1}{2}(\langle a_{i_k},x_k\rangle-b_{i_k})^2
		+\langle x_{k+1}^*-x_{k}^*,x_k^c-\hat{x}\rangle.
		\end{equation}
		For Quantile-RaSK with inexact step,
		we have $x_{k+1}^*-x_{k}^*=-(\langle a_{i_k},x_k\rangle-b_{i_k})a_{i_k}$, obtained from the 8-th step of Algorithm 1. Hence,
		(\ref{3:7}) can be rewritten as
		\begin{align}
		\label{3:9}
		&~~~~D_f^{x_{k+1}^*}(x_{k+1},\hat{x})\nonumber\\
		&\leq
		D_f^{x_{k}^*}(x_{k},\hat{x})
		+
		\frac{1}{2}(\langle a_{i_k},x_k\rangle-b_{i_k})^2
		-(\langle a_{i_k},x_k\rangle-b_{i_k})(\langle a_{i_k},x_k\rangle-\tilde{b}_{i_k}).
		\end{align}
		Now fix the values of the indices $i_0,\cdots,i_{k-1}$ and consider only $i_k$ as a random variable with values in $\{1,\cdots,m\}$. According to $|\langle a_{i_k},x_k\rangle-b_{i_k}|\leq Q_k$, we have
		\begin{align}
		\label{3:10}
		&~~~~\mathbb{E}_k\left(D_f^{x_{k+1}^*}(x_{k+1},\hat{x})|i\in S\right)
		\leq
		D_f^{x_{k}^*}(x_{k},\hat{x})
		+\frac{1}{2}Q_k^2+Q_k\cdot\mathbb{E}_k\left(\langle a_i,x_k-\hat{x}\rangle|i\in S\right),
		\end{align}
		Combining (\ref{3:30}) and (\ref{3:10}), we have that
		\begin{align}
		\label{3:35}
		&~~~~\mathbb{E}_k\left(D_f^{x_{k+1}^*}(x_{k+1},\hat{x})|i\in S\right)\nonumber\\
		&\leq
		D_f^{x_{k}^*}(x_{k},\hat{x})
		+\frac{1}{2}\left(\frac{1-\beta}{1-\beta-q}\right)^2\|r\|_{\infty}^2\nonumber\\
		&~~~+\frac{1-\beta}{1-\beta-q}\left(\frac{1}{\sqrt{|S|}\sqrt{1-\beta}\sqrt{m}}+\frac{1}{2}\cdot\frac{1}{(1-\beta-q)m}
		\right)\sigma_{\max}^2\|x_k-\hat{x}\|^2\nonumber\\
		&
		~~~+\frac{1-\beta}{1-\beta-q}\left(
		\frac{\sqrt{1-\beta}}{(1-\beta-q)\sqrt{m}}+\frac{1}{\sqrt{|S|}}\right)\sigma_{\max}\|x_k-\hat{x}\|\cdot \|r\|_{\infty}
		.
		\end{align}
		To handle the $\|x_k-\hat{x}\|\cdot \|r\|_{\infty}$ term we split into two
		cases: $\|x_k-\hat{x}\|\geq\sqrt{n} \|r\|_{\infty}$ and $\|x_k-\hat{x}\|\leq\sqrt{n} \|r\|_{\infty}$. It is easy to obtain that
		\begin{align}
		\label{3:38}
		&~~~~\mathbb{E}_k\left(D_f^{x_{k+1}^*}(x_{k+1},\hat{x})|i\in S\right)\nonumber\\
		&\leq
		\left(1+\frac{c_{A,\beta,q}'}{\sqrt{|S|}}+C_{A,\beta,q}\right)
		D_f^{x_{k}^*}(x_{k},\hat{x})
		+\left(\frac{d_{A,\beta,q}}{\sqrt{|S|}}+D_{A,\beta,q}\right)
		\|r\|_{\infty}^2,
		\end{align}
		where
		$$c_{A,\beta,q}'=\frac{2\sigma_{\max}^2\sqrt{1-\beta}}{\sqrt{m}(1-\beta-q)}+\frac{2\sigma_{\max}(1-\beta)}{\sqrt{n}(1-\beta-q)},
		C_{A,\beta,q}=
		\frac{\sigma_{\max}^2(1-\beta)}{m(1-\beta-q)^2}
		+
		\frac{2(1-\beta)^{\frac{3}{2}}\sigma_{\max}}{\sqrt{mn}(1-\beta-q)^2},$$	
		$$d_{A,\beta,q}=\frac{(1-\beta)\sqrt{n}}{1-\beta-q}\sigma_{\max},D_{A,\beta,q}=\frac{(1-\beta)^{\frac{3}{2}}\sqrt{n}}{(1-\beta-q)^2\sqrt{m}}\sigma_{\max}+\frac{1}{2}\left(\frac{1-\beta}{1-\beta-q}\right)^2.$$
		Finally, combining (\ref{3:2}), (\ref{3:32}) and (\ref{3:38}) we have
		\begin{align}
		\mathbb{E}_k(D_f^{x_{k+1}^*}(x_{k+1},\hat{x}))
		\leq
		\left(1-C_1\right)D_f^{x_{k}^*}(x_{k},\hat{x})
		+C_2\|r\|_{\infty}^2,
		\end{align}
		where
		\begin{eqnarray}
		&C_1=\frac{q-\beta}{2q^2} \frac{|\hat{x}|_{\min}}{|\hat{x}|_{\min}+2\lambda} \frac{\tilde{\sigma}_{q-\beta,\min}^2}{m}
		-
		\frac{2\sqrt{\beta(1-\beta)}}{q(1-\beta-q)}\left(\frac{\sigma_{\max}^2}{m}+\frac{\sigma_{\max}}{\sqrt{mn}}\right)
		-
		\frac{\beta(1-\beta)}{q(1-\beta-q)^2}\left(\frac{\sigma_{\max}^2}{m}+\frac{\sqrt{1-\beta}\sigma_{\max }}{\sqrt{mn}} \right),\nonumber\\
		&C_2=\frac{\sqrt{\beta}(1-\beta)}{q(1-q-\beta)}\left(1+\frac{\sqrt{\beta(1-\beta)}}{1-q-\beta}\right)\sqrt{\frac{n}{m}}\sigma_{\max}
		+\frac{1}{2}\frac{\beta(1-\beta)^2}{q(1-\beta-q)^2}
		+\frac{1}{2}.\nonumber
		\end{eqnarray}
		We consider all indices $i_0,i_1,\cdots,i_k$ as random variables, and take full expectation on both sides. Thus,
		\begin{align}
		\label{3:40}
		\mathbb{E}(D_f^{x_{k+1}^*}(x_{k+1},\hat{x}))
		\leq
		\left(1-C_1\right)\mathbb{E}(D_f^{x_{k}^*}(x_{k},\hat{x}))
		+C_2\|r\|_{\infty}^2,
		\end{align}
		(b) In the exact-step case, first we consider the uncorrupted equations indexed $i\in B\backslash S$. Since $b_{B\backslash S}^c=0$, the equations indexed $i\in B\backslash S$ satisfy
		\begin{eqnarray}
		A_{B\backslash S}x=\tilde{b}_{B\backslash S}=b_{B\backslash S}-r_{B\backslash S}.\nonumber
		\end{eqnarray}
		It follows from the convergence rate of ERaSK in noisy case (\ref{rasknoisy:2}) that
		\begin{eqnarray}
		\label{3:41}
		&~~~~\mathbb{E}(D_f^{x_{k+1}^*}(x_{k+1},\hat{x})|i\in B\backslash S)\nonumber\\
		&\leq
		\left(1-\frac{1}{2}\frac{\tilde{\sigma}_{q-\beta,\min}^2}{qm}\frac{|\hat{x}|_{\min}}{|\hat{x}|_{\min}+2\lambda}\right)
		D_f^{x_{k}^*}(x_{k},\hat{x})
		+\frac{1}{2}\frac{qm}{|B\backslash S|}\|r\|_{\infty}^2
		+\frac{2\sqrt{qm}}{|B\backslash S|}\|r\|_{\infty}\|A\|_{1,2}.
		\end{eqnarray}
		
		Second, we consider the conditional expectation in $S$. In this case, we have
		\begin{eqnarray}
		A_Sx=\tilde{b_S}, b_S=\tilde{b_S}+b_S^c+r_S.\nonumber
		\end{eqnarray}
		For Quantile-RaSK with exact step,
		we have $x_{k}^*=x_{k}+\lambda s_k$ with $\|s_k\|_{\infty}\leq1$ and $\|s_{k+1}\|_{\infty}\leq1$; then $x_{k+1}^*-x_{k}^*=(x_{k+1}-x_{k})+\lambda(s_{k+1}-s_{k})$. Note that
		the exact linesearch guarantees $\langle a_{i_k},x_{k+1}\rangle=b_{i_k}$.
		Thus,
		(\ref{3:7}) can be rewritten as
		\begin{align}
		\label{3:22}
		D_f^{x_{k+1}^*}(x_{k+1},\hat{x})
		&\leq
		D_f^{x_{k}^*}(x_{k},\hat{x})
		+
		\lambda \langle s_{k+1}-s_{k},a_{i_k}\rangle
		(b_{i_k}-\tilde{b}_{i_k})\nonumber\\
		&~~~+
		\frac{1}{2}(\langle a_{i_k},x_k\rangle-b_{i_k})^2
		-
		(\langle a_{i_k},x_k\rangle-b_{i_k})
		(\langle a_{i_k},x_k\rangle-\tilde{b}_{i_k}).
		\end{align}
		Viewing $i_k$ as a random variable with fixed $i_0,\ldots,i_{k-1}$, yields
		\begin{align}
		\label{3:24}
		\mathbb{E}_k(\lambda \langle s_{k+1}-s_{k},a_i\rangle
		(b_i-\tilde{b}_i)|i\in S)
		&\leq
		2\lambda\mathbb{E}_k(\|a_i\|_1\cdot|b_i-\tilde{b}_i||i\in S)\nonumber\\
		&\leq
		\frac{2\lambda}{|S|}\sum_{i\in S}(\|a_i\|_1\cdot|b_i-\tilde{b}_i|)\nonumber\\
		&\leq
		\frac{2\lambda}{|S|}\|b-\tilde{b}\|\cdot\|A\|_{1,2}.
		\end{align}
		And recall the conclusion (\ref{3:38}) in (a), we obtain
		\begin{align}
		\label{3:44}
		&~~~\mathbb{E}_k\left(D_f^{x_{k+1}^*}(x_{k+1},\hat{x})|i\in S\right)\nonumber\\
		&\leq
		\left(1+\frac{c_{A,\beta,q}'}{\sqrt{|S|}}+C_{A,\beta,q}\right)
		D_f^{x_{k}^*}(x_{k},\hat{x})
		+\left(\frac{d_{A,\beta,q}}{\sqrt{|S|}}+D_{A,\beta,q}\right)
		\|r\|_{\infty}^2\nonumber\\
		&~~~
		+\frac{2\lambda}{|S|}\|b-\tilde{b}\|\cdot\|A\|_{1,2}.
		\end{align}
		Finally, combining all ingredients: (\ref{3:2}), (\ref{3:41}) and (\ref{3:44}), and taking full expectation, we have that
		\begin{align}
		\label{3:46}
		&~~~~\mathbb{E}(D_f^{x_{k+1}^*}(x_{k+1},\hat{x}))\nonumber\\
		&\leq
		\left(1-C_1\right)\mathbb{E}D_f^{x_{k}^*}(x_{k},\hat{x})
		+C_2\|r\|_{\infty}^2
		+\frac{2}{\sqrt{qm}}\|r\|_{\infty}\cdot\|A\|_{1,2}
		+\frac{2\lambda}{qm}\|b-\tilde{b}\|\cdot\|A\|_{1,2}.\nonumber
		\end{align}
		(c) To ensure the decay in expectation, we require
		$$
		\left(
		\frac{2\sqrt{1-\beta}}{(1-\beta-q)\sqrt{\beta}m}+\frac{1-\beta}{(1-\beta-q)^2m}
		\right)\cdot\sigma_{\max}
		+
		\left(
		\frac{2(1-\beta)}{(1-\beta-q)\sqrt{mn\beta}}+
		\frac{2(1-\beta)^{\frac{3}{2}}}{(1-\beta-q)^2\sqrt{mn}}
		\right)\nonumber
		$$
		$$
		<
		\frac{1}{2}\cdot\frac{q-\beta}{\beta}\cdot\frac{|\hat{x}|_{\min}}{|\hat{x}|_{\min}+2\lambda}\cdot\frac{\tilde{\sigma}_{q-\beta,\min}^2}{\sigma_{\max}},\nonumber
		$$
		which holds for a small enough parameter $\beta$ since the left-hand side of it tends to zero as $\beta$ tends to zero.
		Therefore, we obtain the conclusion.
	\end{proof}

	\begin{proof}[The proof of Theorem \ref{th:noisy-general}.]
		Using the constant stepsize, the update in Algorithm 2 is as follows:
		\begin{equation}
		\begin{aligned}
		\label{th2:1}
		&x_{k+1}^*=x_k^*-\frac{w}{\eta}\sum_{i\in T} (\langle a_i,x_k\rangle-b_i) a_{i},\\
		&x_{k+1}=\mathcal{S}_{\lambda}(x_{k+1}^*).
		\end{aligned}
		\end{equation}
		Denote
		\begin{align}
		\label{th2:3}
		x_{k}^{\delta}:=\hat{x}+\frac{w}{\eta}\sum_{i\in T} (b_i-\tilde{b}_i) a_{i}.
		\end{align}
		Now use Lemma \ref{lemma2} with
		$f(x)=\lambda\|x\|_{1}+\frac{1}{2}\|x\|_2^2$,
		$\Phi (x)=\langle x_k^*-x_{k+1}^*,x-x_k\rangle$
		and $y=x_k^{\delta}$,
		and it holds that
		\begin{equation*}
		\begin{aligned}
		&~~~~D_f^{x_{k+1}^*}\left(x_{k+1}, x_k^{\delta}\right)\\
		&\leq
		D_f^{x_k^*}\left(x_k, x_k^{\delta}\right)
		+\langle x_k^{*}-x_{k+1}^{*},x_k^{\delta}-x_k\rangle
		-\langle x_k^*-x_{k+1}^*,x_{k+1}-x_k\rangle
		-D_f^{x_{k+1}^*}\left(x_{k},x_{k+1}\right)\\
		&\leq
		D_f^{x_k^*}\left(x_k, x_k^{\delta}\right)
		+\langle x_k^*-x_{k+1}^*,x_k^{\delta}-x_k\rangle
		+\| x_k^*-x_{k+1}^*\|\cdot \|x_{k+1}-x_k\|
		-\frac{1}{2} \|x_k-x_{k+1}\|^2\\
		&\leq
		D_f^{x_k^*}\left(x_k, x_k^{\delta}\right)
		+\langle x_k^*-x_{k+1}^*,x_k^{\delta}-x_k\rangle
		+\frac{1}{2} \|x_k^*-x_{k+1}^*\|^2,
		\end{aligned}
		\end{equation*}
		Unfolding the expression of $x_k^{\delta}$ in (\ref{th2:3}), we obtain
		\begin{equation*}
		\begin{aligned}
		&~~~~D_f^{x_{k+1}^*}\left(x_{k+1}, \hat{x}\right)\\
		&\leq
		D_f^{x_k^*}\left(x_k, \hat{x}\right)
		+\langle x_{k}^*-x_{k+1}^*,\hat{x}-x_k\rangle
		+\frac{1}{2} \|x_k^*-x_{k+1}^*\|^2,
		\end{aligned}
		\end{equation*}
		We divide it into two steps: the scalar product and the quadratic term.\\
		\textbf{Step 1:} For the inner product term, we can derive that 
		\begin{equation}
		\begin{aligned}
		\label{eq5:1}
		&~~~~\langle x_{k}^*-x_{k+1}^*,\hat{x}-x_k\rangle\\
		&=
		-\langle\frac{w}{\eta}\sum_{i\in T} (\langle a_i,x_k\rangle-b_i) a_{i},x_k-\hat{x}\rangle\\
		&=
		-\langle\frac{w}{\eta}\sum_{i\in T_1} (\langle a_i,x_k\rangle-b_i) a_{i},x_k-\hat{x}\rangle
		-\langle\frac{w}{\eta}\sum_{i\in T_2} (\langle a_i,x_k\rangle-b_i) a_{i},x_k-\hat{x}\rangle\\
		\end{aligned}
		\end{equation}
		
		From the first term in (\ref{eq5:1}), we obtain
		\begin{align}
		\label{eq5:2}
		&~~~~-\langle\frac{w}{\eta}\sum_{i\in T_1}(\langle a_i,x_k\rangle-b_i) a_{i},x_k-\hat{x}\rangle\nonumber\\
		&=
		-\langle\frac{w}{\eta}\sum_{i\in T_1}  (\langle a_i,x_k\rangle-\tilde{b}_i) a_{i},x_k-\hat{x}\rangle
		+\langle\frac{w}{\eta}\sum_{i\in T_1}  r_i a_{i},x_k-\hat{x}\rangle\nonumber\\
		&\leq
		-\frac{w}{\eta}\|\sum_{i\in T_1}a_i^Ta_i\|\cdot\|x_k-\hat{x}\|^2
		+\frac{w}{\eta}\cdot\|r\|_{\infty}\cdot\|\sum_{i\in T_1}a_i\|\cdot\|x_k-\hat{x}\|\nonumber\\
		&\leq
		-\frac{w}{qm}\sigma_{q-\beta,\min}^2\cdot\|x_k-\hat{x}\|^2
		+
		\frac{w}{\sqrt{qm}}\sigma_{\max}\cdot\|r\|_{\infty}\cdot\|x_k-\hat{x}\|.
		\end{align}
		
		From the second term in (\ref{eq5:1}), we get
		\begin{align}
		\label{eq5:3}
		&~~~~\langle\frac{w}{\eta}\sum_{i\in T_2} (\langle a_i,x_k\rangle-b_i) a_{i},x_k-\hat{x}\rangle\nonumber\\
		&\leq
		\frac{w}{\eta}Q_k\|\sum_{i\in T_2}a_{i}\|\cdot\|x_k-\hat{x}\|\nonumber\\
		&\leq
		\frac{w}{\eta}Q_k\sqrt{| T_2|}\sigma_{\max}\cdot\|x_k-\hat{x}\|\nonumber\\
		&\leq
		\frac{w\sqrt{(1-\beta)\beta}}{(1-\beta-q)qm}\sigma_{\max}^2\cdot\|x_k-\hat{x}\|^2
		+
		\frac{1-\beta}{1-\beta-q}
		\frac{w\sqrt{\beta}}{q\sqrt{m}}\sigma_{\max}\cdot\|r\|_{\infty}\cdot\|x_k-\hat{x}\|.
		\end{align}
		The last inequality makes use of Lemma \ref{lemma1}.
		Combining (\ref{eq5:1}), (\ref{eq5:2}) and (\ref{eq5:3}), we obtain
		\begin{align}
		\label{eq5:7}
		&~~~~\langle x_{k}^*-x_{k+1}^*,\hat{x}-x_k\rangle\nonumber\\
		&\leq
		\left(
		-\frac{w}{qm}\sigma_{q-\beta,\min}^2
		+\frac{w\sqrt{(1-\beta)\beta}}{(1-\beta-q)qm}\sigma_{\max}^2
		\right)\|x_k-\hat{x}\|^2\nonumber\\
		&~~~~+
		\left(
		\frac{w}{\sqrt{qm}}\sigma_{\max}
		+
		\frac{1-\beta}{1-\beta-q}
		\frac{w\sqrt{\beta}}{q\sqrt{m}}\sigma_{\max}
		\right)\|r\|_{\infty}\cdot\|x_k-\hat{x}\|.
		\end{align}
		\textbf{Step 2:} For the quadratic term, we have
		\begin{align}
		&~~~~\|x_k^*-x_{k+1}^*\|^2\nonumber\\
		&=\|\frac{w}{\eta}\sum_{i\in T} (\langle a_i,x_k\rangle-b_i) a_{i}\|^2\nonumber\\
		&=\|\frac{w}{\eta}\sum_{i\in T_1} (\langle a_i,x_k\rangle-b_i) a_{i}+\frac{w}{\eta}\sum_{i\in T_2} (\langle a_i,x_k\rangle-b_i) a_{i}\|^2,\nonumber\\
		&=\|u+v\|^2,\nonumber\\
		&\leq (\|u\|+\|v\|)^2,
		\end{align}
		where $u=\frac{w}{\eta}\sum_{i\in T_1}  (\langle a_i,x_k\rangle-b_i) a_{i},v=\frac{w}{\eta}\sum_{i\in T_2}  (\langle a_i,x_k\rangle-b_i) a_{i}$.
		
		We have
		\begin{align}
		\label{eq5:4}
		\|u\|^2
		&=\|\frac{w}{\eta}\sum_{i\in T_1} (\langle a_i,x_k\rangle-b_i) a_{i}\|^2\nonumber\\
		&\leq
		\frac{w^2}{\eta^2}\|\sum_{i\in T_1} (\langle a_i,x_k\rangle-\tilde{b}_i-r_i) a_{i}\|^2\nonumber\\
		&=
		\frac{w^2}{\eta^2}\|\sum_{i\in T_1} a_ia_i^T(x_k-\hat{x})-\sum_{i\in T_1}a_ir_i\|^2\nonumber\\
		&\leq
		\frac{w^2}{\eta^2}\left(
		\|\sum_{i\in T_1} a_ia_i^T(x_k-\hat{x})\|^2
		-2\langle
		\sum_{i\in T_1} a_ia_i^T(x_k-\hat{x}),\sum_{i\in T_1}a_ir_i\rangle+
		\|\sum_{i\in T_1}a_ir_i\|^2
		\right)\nonumber\\
		&=
		\frac{w^2}{\eta^2}
		\left(
		\|A_{ T_1}^T A_{ T_1}(x_k-\hat{x})\|^2
		-
		2\langle A_{ T_1}^TA_{ T_1}(x_k-\hat{x}),\sum_{i\in T_1}a_ir_i\rangle
		+
		\|\sum_{i\in T_1}a_ir_i\|^2
		\right)\nonumber\\
		&\leq
		\frac{w^2}{\eta^2}
		\left(
		\|A_{ T_1}^T A_{ T_1}(x_k-\hat{x})\|
		+\|\sum_{i\in T_1}a_ir_i\|
		\right)^2\nonumber\\
		&\leq
		\frac{w^2\sigma_{\max}^2}{q^2m^2}\left(
		\sigma_{\max}\|x_k-\hat{x}\|+\sqrt{qm}\|r\|_{\infty}
		\right)^2.
		\end{align}
		and
		\begin{align}
		\label{eq5:5}
		\|v\|^2
		&=\|\frac{w}{\eta}\sum_{i\in T_2} (\langle a_i,x_k\rangle-b_i) a_{i}\|^2\nonumber\\
		&\leq
		\frac{w^2}{\eta^2}\cdot Q_k^2\cdot\|\sum_{i\in T_2}a_i\|^2\nonumber\\
		&\leq
		\frac{w^2\sigma_{\max}^2\beta}{q^2m}\left(
		\frac{\sqrt{1-\beta}}{(1-\beta-q)\sqrt{m}}\sigma_{\max}\|x_k-\hat{x}\|+
		\frac{1-\beta}{1-\beta-q}\|r\|_{\infty}
		\right)^2.
		\end{align}
		Bring (\ref{eq5:4}) and (\ref{eq5:5}) together, we obtain
		\begin{align}
		\label{eq5:6}
		&~~~~\|u+v\|^2\nonumber\\
		&\leq
		\frac{w^2\sigma_{\max}^2}{q^2m^2}\left(
		\sigma_{\max}\|x_k-\hat{x}\|+\sqrt{qm}\|r\|_{\infty}
		+
		\frac{\sqrt{(1-\beta)\beta}}{1-\beta-q}\sigma_{\max}\|x_k-\hat{x}\|+
		\frac{(1-\beta)\sqrt{\beta m}}{1-\beta-q}\|r\|_{\infty}
		\right)^2\nonumber\\
		&=
		\frac{w^2\sigma_{\max}^2}{q^2m^2}
		\left[
		\sigma_{\max}
		\left(
		1+\frac{\sqrt{(1-\beta)\beta}}{1-\beta-q}
		\right)\|x_k-\hat{x}\|
		+
		\sqrt{qm}
		\left(
		1+\frac{(1-\beta)\sqrt{\beta}}{(1-\beta-q)\sqrt{q}}
		\right)\|r\|_{\infty}
		\right]^2\nonumber\\
		&\leq
		\frac{w^2}{q^2m^2}\sigma_{\max}^4\left(
		1+	\frac{\sqrt{(1-\beta)\beta}}{1-\beta-q}
		\right)^2\|x_k-\hat{x}\|^2
		+
		\frac{w^2}{qm}\sigma_{\max}^2\left(
		1+\frac{(1-\beta)\sqrt{\beta}}{(1-\beta-q)\sqrt{q}}
		\right)^2\|r\|_{\infty}^2\nonumber\\
		&~~~~+
		\frac{w^2}{(qm)^{\frac{3}{2}}}\sigma_{\max}^3
		\left(
		1+	\frac{\sqrt{(1-\beta)\beta}}{1-\beta-q}
		\right)
		\left(
		1+\frac{(1-\beta)\sqrt{\beta}}{(1-\beta-q)\sqrt{q}}
		\right)
		\|x_k-\hat{x}\|\cdot\|r\|_{\infty}.
		\end{align}
		Combining (\ref{eq5:7}) and (\ref{eq5:6}) yields
		\begin{equation*}
		\begin{aligned}
		&~~~~D_f^{x_{k+1}^*}\left(x_{k+1}, \hat{x}\right)\\
		&\leq
		D_f^{x_k^*}\left(x_k, \hat{x}\right)
		+\langle x_{k}^*-x_{k+1}^*,\hat{x}-x_k\rangle
		+\frac{1}{2} \|x_k^*-x_{k+1}^*\|^2,\nonumber\\
		&\leq
		D_f^{x_k^*}\left(x_k, \hat{x}\right)
		+
		\left(
		-\frac{w}{qm}\sigma_{q-\beta,\min}^2
		+\frac{w\sqrt{(1-\beta)\beta}}{(1-\beta-q)qm}\sigma_{\max}^2
		\right)\|x_k-\hat{x}\|^2\nonumber\\
		&~~~~+
		\left(
		\frac{w}{\sqrt{qm}}\sigma_{\max}
		+
		\frac{1-\beta}{1-\beta-q}
		\frac{w\sqrt{\beta}}{q\sqrt{m}}\sigma_{\max}
		\right)\|r\|_{\infty}\cdot\|x_k-\hat{x}\|\nonumber\\
		&~~~~+
		\frac{w^2}{q^2m^2}\sigma_{\max}^4\left(
		1+	\frac{\sqrt{(1-\beta)\beta}}{1-\beta-q}
		\right)^2\|x_k-\hat{x}\|^2
		+
		\frac{w^2}{qm}\sigma_{\max}^2\left(
		1+\frac{(1-\beta)\sqrt{\beta}}{(1-\beta-q)\sqrt{q}}
		\right)^2\|r\|_{\infty}^2\nonumber\\
		&~~~~+
		\frac{w^2}{(qm)^{\frac{3}{2}}}\sigma_{\max}^3
		\left(
		1+	\frac{\sqrt{(1-\beta)\beta}}{1-\beta-q}
		\right)
		\left(
		1+\frac{(1-\beta)\sqrt{\beta}}{(1-\beta-q)\sqrt{q}}
		\right)
		\|x_k-\hat{x}\|\cdot\|r\|_{\infty}\\
		&=
		D_f^{x_k^*}\left(x_k, \hat{x}\right)
		+
		\left(
		-c_1\frac{w}{m}\sigma_{q-\beta,\min}^2
		+c_2\frac{ w}{m}\sigma_{\max}^2
		+c_3 \frac{w^2}{m^2}\sigma_{\max}^4
		\right)\|x_k-\hat{x}\|^2\nonumber\\
		&~~~~+
		\left(
		c_4\frac{w}{\sqrt{m}}\sigma_{\max}
		+c_5\frac{w^2}{m^{\frac{3}{2}}}\sigma_{\max}^3
		\right)
		\|r\|_{\infty}\cdot\|x_k-\hat{x}\|
		+
		c_6\frac{w^2}{m}\sigma_{\max}^2\|r\|_{\infty}^2,\nonumber\\
		\end{aligned}
		\end{equation*}
		where $c_i,i=1,\cdots,6$ are positive constants only related to $q$ and $\beta$ and can be directly obtained.
		For the term $\|r\|_{\infty}\cdot\|x_k-\hat{x}\|$, we use average inequality
		$$\|r\|_{\infty}\cdot\|x_k-\hat{x}\|\leq \frac{1}{2\sqrt{n}}\|x_k-\hat{x}\|^2+\frac{\sqrt{n}}{2}\|r\|_{\infty}^2.$$
		Denote $\alpha=\frac{|\hat{x}|_{\min}}{ |\hat{x}|_{\min} +2 \lambda}$, and use Lemma \ref{lemma3}, we have
		\begin{align}
		\label{appen:1}
		&~~~~D_f^{x_{k+1}^*}\left(x_{k+1}, \hat{x}\right)\nonumber\\
		&\leq
		D_f^{x_k^*}\left(x_k, \hat{x}\right)
		+
		\left(
		-c_1\frac{w}{m}\sigma_{q-\beta,\min}^2
		+c_2\frac{w}{m}\sigma_{\max}^2
		+c_3\frac{w^2}{m^2}\sigma_{\max}^4
		+\frac{c_4}{2}\frac{w}{\sqrt{mn}}\sigma_{\max}
		\right)\|x_k-\hat{x}\|^2\nonumber\\
		&~~~~~~
		+\frac{c_5}{2}\frac{w^2}{m^{3/2}n^{1/2}}\sigma_{\max}^3\|x_k-\hat{x}\|^2
		+
		\left(
		\frac{c_4}{2}\frac{\sqrt{n}w}{\sqrt{m}}\sigma_{\max}+\frac{c_5}{2}\frac{\sqrt{n}w^2}{m^{3/2}}\sigma_{\max}^3+c_6\frac{w^2}{m}\sigma_{\max}^2
		\right)\|r\|_{\infty}^2
		\nonumber\\
		&\leq
		\left[
		1-\alpha\left(
		c_1\frac{w}{m}\frac{\tilde{\sigma}_{\min }^2\sigma_{q-\beta,\min}^2}{\sigma_{\max}^2}
		-c_2\frac{w}{m}\tilde{\sigma}_{\min }^2
		-c_3\frac{w^2}{m^2}\tilde{\sigma}_{\min }^2\sigma_{\max}^2
		-\frac{c_4}{2}\frac{w\tilde{\sigma}_{\min }^2}{\sqrt{mn}\sigma_{\max}}
		\right)
		\right]D_f^{x_k^*}\left(x_k, \hat{x}\right)\nonumber\\
		&~~~~~~
		-\frac{c_5}{2}\frac{w^2\tilde{\sigma}_{\min }^2}{m^{3/2}n^{1/2}}\sigma_{\max}D_f^{x_k^*}\left(x_k, \hat{x}\right)
		+
		\left(
		\frac{c_4}{2}\frac{\sqrt{n}w}{\sqrt{m}}\sigma_{\max}+\frac{c_5}{2}\frac{\sqrt{n}w^2}{m^{3/2}}\sigma_{\max}^3+c_6\frac{w^2}{m}\sigma_{\max}^2
		\right)\|r\|_{\infty}^2\nonumber
		\\
		&=(1-c_1^*w+c_2^*w^2)D_f^{x_k^*}\left(x_k, \hat{x}\right)
		+(c_3^*w+c_4^*w^2)\|r\|_{\infty}^2,
		\end{align}
		where
		$$c_1^*=c_1\alpha\frac{\tilde{\sigma}_{\min }^2\sigma_{q-\beta,\min}^2}{ m\sigma_{\max}^2}-c_2\alpha\frac{\tilde{\sigma}_{\min }^2}{ m}-c_4\alpha\frac{\tilde{\sigma}_{\min }^2}{2\sqrt{mn}\sigma_{\max}},~
		c_2^*=c_3\alpha\frac{\tilde{\sigma}_{\min }^2\sigma_{\max}^2}{ m^2}+c_5\alpha\frac{\tilde{\sigma}_{\min }^2\sigma_{\max}}{2 m^{3/2}n^{1/2}},$$
		$$c_3^*=c_4\frac{\sqrt{n}\sigma_{\max}}{2\sqrt{m}},~c_4^*=c_5\frac{\sqrt{n}\sigma_{\max}^3}{2m^{3/2}}+c_6\frac{\sigma_{\max}^2}{m}.$$

	\end{proof}
\end{appendices}

\bibliography{ref}


\begin{thebibliography}{39}
\ifx \bisbn   \undefined \def \bisbn  #1{ISBN #1}\fi
\ifx \binits  \undefined \def \binits#1{#1}\fi
\ifx \bauthor  \undefined \def \bauthor#1{#1}\fi
\ifx \batitle  \undefined \def \batitle#1{#1}\fi
\ifx \bjtitle  \undefined \def \bjtitle#1{#1}\fi
\ifx \bvolume  \undefined \def \bvolume#1{\textbf{#1}}\fi
\ifx \byear  \undefined \def \byear#1{#1}\fi
\ifx \bissue  \undefined \def \bissue#1{#1}\fi
\ifx \bfpage  \undefined \def \bfpage#1{#1}\fi
\ifx \blpage  \undefined \def \blpage #1{#1}\fi
\ifx \burl  \undefined \def \burl#1{\textsf{#1}}\fi
\ifx \doiurl  \undefined \def \doiurl#1{\url{https://doi.org/#1}}\fi
\ifx \betal  \undefined \def \betal{\textit{et al.}}\fi
\ifx \binstitute  \undefined \def \binstitute#1{#1}\fi
\ifx \binstitutionaled  \undefined \def \binstitutionaled#1{#1}\fi
\ifx \bctitle  \undefined \def \bctitle#1{#1}\fi
\ifx \beditor  \undefined \def \beditor#1{#1}\fi
\ifx \bpublisher  \undefined \def \bpublisher#1{#1}\fi
\ifx \bbtitle  \undefined \def \bbtitle#1{#1}\fi
\ifx \bedition  \undefined \def \bedition#1{#1}\fi
\ifx \bseriesno  \undefined \def \bseriesno#1{#1}\fi
\ifx \blocation  \undefined \def \blocation#1{#1}\fi
\ifx \bsertitle  \undefined \def \bsertitle#1{#1}\fi
\ifx \bsnm \undefined \def \bsnm#1{#1}\fi
\ifx \bsuffix \undefined \def \bsuffix#1{#1}\fi
\ifx \bparticle \undefined \def \bparticle#1{#1}\fi
\ifx \barticle \undefined \def \barticle#1{#1}\fi
\bibcommenthead
\ifx \bconfdate \undefined \def \bconfdate #1{#1}\fi
\ifx \botherref \undefined \def \botherref #1{#1}\fi
\ifx \url \undefined \def \url#1{\textsf{#1}}\fi
\ifx \bchapter \undefined \def \bchapter#1{#1}\fi
\ifx \bbook \undefined \def \bbook#1{#1}\fi
\ifx \bcomment \undefined \def \bcomment#1{#1}\fi
\ifx \oauthor \undefined \def \oauthor#1{#1}\fi
\ifx \citeauthoryear \undefined \def \citeauthoryear#1{#1}\fi
\ifx \endbibitem  \undefined \def \endbibitem {}\fi
\ifx \bconflocation  \undefined \def \bconflocation#1{#1}\fi
\ifx \arxivurl  \undefined \def \arxivurl#1{\textsf{#1}}\fi
\csname PreBibitemsHook\endcsname

\bibitem[\protect\citeauthoryear{Lorenz et~al.}{2014}]{lorenz2014sparse}
\begin{bchapter}
\bauthor{\bsnm{Lorenz}, \binits{D.A.}},
\bauthor{\bsnm{Wenger}, \binits{S.}},
\bauthor{\bsnm{Sch{\"o}pfer}, \binits{F.}},
\bauthor{\bsnm{Magnor}, \binits{M.}}:
\bctitle{A sparse {K}aczmarz solver and a linearized {B}regman method for
  online compressed sensing}.
In: \bbtitle{2014 IEEE International Conference on Image Processing (ICIP)},
pp. \bfpage{1347}--\blpage{1351}
(\byear{2014}).
\doiurl{10.1109/ICIP.2014.7025269} .
\bcomment{IEEE}
\end{bchapter}
\endbibitem

\bibitem[\protect\citeauthoryear{Tan and Vershynin}{2019}]{tan2019phase}
\begin{barticle}
\bauthor{\bsnm{Tan}, \binits{Y.S.}},
\bauthor{\bsnm{Vershynin}, \binits{R.}}:
\batitle{Phase retrieval via randomized {K}aczmarz: theoretical guarantees}.
\bjtitle{Information and Inference: A Journal of the IMA}
\bvolume{8}(\bissue{1}),
\bfpage{97}--\blpage{123}
(\byear{2019})
\end{barticle}
\endbibitem

\bibitem[\protect\citeauthoryear{Xian et~al.}{2022}]{xian2022randomized}
\begin{barticle}
\bauthor{\bsnm{Xian}, \binits{Y.}},
\bauthor{\bsnm{Liu}, \binits{H.}},
\bauthor{\bsnm{Tai}, \binits{X.}},
\bauthor{\bsnm{Wang}, \binits{Y.}}:
\batitle{Randomized {K}aczmarz method for single particle x-ray image phase
  retrieval}.
\bjtitle{arXiv preprint arXiv:2207.04736}
(\byear{2022})
\doiurl{10.1007/978-3-030-98661-2_112}
\end{barticle}
\endbibitem

\bibitem[\protect\citeauthoryear{R{\"o}mer et~al.}{2021}]{romer2021randomized}
\begin{bchapter}
\bauthor{\bsnm{R{\"o}mer}, \binits{P.}},
\bauthor{\bsnm{Filbir}, \binits{F.}},
\bauthor{\bsnm{Krahmer}, \binits{F.}}:
\bctitle{On the randomized {K}aczmarz algorithm for phase retrieval}.
In: \bbtitle{2021 55th Asilomar Conference on Signals, Systems, and Computers},
pp. \bfpage{847}--\blpage{851}
(\byear{2021}).
\doiurl{10.1109/IEEECONF53345.2021.9723291} .
\bcomment{IEEE}
\end{bchapter}
\endbibitem

\bibitem[\protect\citeauthoryear{Chen and Qin}{2021}]{chen2021regularized}
\begin{barticle}
\bauthor{\bsnm{Chen}, \binits{X.}},
\bauthor{\bsnm{Qin}, \binits{J.}}:
\batitle{Regularized {K}aczmarz algorithms for tensor recovery}.
\bjtitle{SIAM Journal on Imaging Sciences}
\bvolume{14}(\bissue{4}),
\bfpage{1439}--\blpage{1471}
(\byear{2021})
\doiurl{10.1137/21M1398562}
\end{barticle}
\endbibitem

\bibitem[\protect\citeauthoryear{Du and Sun}{2021}]{du2021randomized}
\begin{botherref}
\oauthor{\bsnm{Du}, \binits{K.}},
\oauthor{\bsnm{Sun}, \binits{X.-H.}}:
Randomized regularized extended {K}aczmarz algorithms for tensor recovery.
arXiv preprint arXiv:2112.08566
(2021)
\end{botherref}
\endbibitem

\bibitem[\protect\citeauthoryear{Gordon et~al.}{1970}]{gordon1970algebraic}
\begin{barticle}
\bauthor{\bsnm{Gordon}, \binits{R.}},
\bauthor{\bsnm{Bender}, \binits{R.}},
\bauthor{\bsnm{Herman}, \binits{G.T.}}:
\batitle{Algebraic reconstruction techniques (art) for three-dimensional
  electron microscopy and x-ray photography}.
\bjtitle{Journal of theoretical Biology}
\bvolume{29}(\bissue{3}),
\bfpage{471}--\blpage{481}
(\byear{1970})
\doiurl{10.1016/0022-5193(70)90109-8}
\end{barticle}
\endbibitem

\bibitem[\protect\citeauthoryear{Jarman and
  Needell}{2021}]{jarman2021quantilerk}
\begin{bchapter}
\bauthor{\bsnm{Jarman}, \binits{B.}},
\bauthor{\bsnm{Needell}, \binits{D.}}:
\bctitle{Quantilerk: Solving large-scale linear systems with corrupted, noisy
  data}.
In: \bbtitle{2021 55th Asilomar Conference on Signals, Systems, and Computers},
pp. \bfpage{1312}--\blpage{1316}
(\byear{2021}).
\doiurl{10.1109/IEEECONF53345.2021.9723338} .
\bcomment{IEEE}
\end{bchapter}
\endbibitem

\bibitem[\protect\citeauthoryear{Needell}{2010}]{needell2010randomized}
\begin{barticle}
\bauthor{\bsnm{Needell}, \binits{D.}}:
\batitle{Randomized {K}aczmarz solver for noisy linear systems}.
\bjtitle{BIT Numerical Mathematics}
\bvolume{50},
\bfpage{395}--\blpage{403}
(\byear{2010})
\doiurl{10.1007/s10543-010-0265-5}
\end{barticle}
\endbibitem

\bibitem[\protect\citeauthoryear{Sch{\"o}pfer and
  Lorenz}{2019}]{schopfer2019linear}
\begin{barticle}
\bauthor{\bsnm{Sch{\"o}pfer}, \binits{F.}},
\bauthor{\bsnm{Lorenz}, \binits{D.A.}}:
\batitle{Linear convergence of the randomized sparse {K}aczmarz method}.
\bjtitle{Mathematical Programming}
\bvolume{173}(\bissue{1}),
\bfpage{509}--\blpage{536}
(\byear{2019})
\doiurl{10.1007/s10107-017-1229-1}
\end{barticle}
\endbibitem

\bibitem[\protect\citeauthoryear{Yuan et~al.}{2022a}]{yuan2022sparse}
\begin{barticle}
\bauthor{\bsnm{Yuan}, \binits{Z.}},
\bauthor{\bsnm{Zhang}, \binits{H.}},
\bauthor{\bsnm{Wang}, \binits{H.}}:
\batitle{Sparse sampling {K}aczmarz-{M}otzkin method with linear convergence}.
\bjtitle{Mathematical Methods in the Applied Sciences}
\bvolume{45}(\bissue{7}),
\bfpage{3463}--\blpage{3478}
(\byear{2022})
\doiurl{10.1002/mma.7990}
\end{barticle}
\endbibitem

\bibitem[\protect\citeauthoryear{Yuan et~al.}{2022b}]{yuan2022adaptively}
\begin{barticle}
\bauthor{\bsnm{Yuan}, \binits{Z.-Y.}},
\bauthor{\bsnm{Zhang}, \binits{L.}},
\bauthor{\bsnm{Wang}, \binits{H.}},
\bauthor{\bsnm{Zhang}, \binits{H.}}:
\batitle{Adaptively sketched {B}regman projection methods for linear systems}.
\bjtitle{Inverse Problems}
\bvolume{38}(\bissue{6}),
\bfpage{065005}
(\byear{2022})
\doiurl{10.1088/1361-6420/ac5f76}
\end{barticle}
\endbibitem

\bibitem[\protect\citeauthoryear{Zhang et~al.}{2022}]{zhang2022weighted}
\begin{barticle}
\bauthor{\bsnm{Zhang}, \binits{L.}},
\bauthor{\bsnm{Yuan}, \binits{Z.}},
\bauthor{\bsnm{Wang}, \binits{H.}},
\bauthor{\bsnm{Zhang}, \binits{H.}}:
\batitle{A weighted randomized sparse {K}aczmarz method for solving linear
  systems}.
\bjtitle{Computational and Applied Mathematics}
\bvolume{41}(\bissue{8}),
\bfpage{1}--\blpage{18}
(\byear{2022})
\end{barticle}
\endbibitem

\bibitem[\protect\citeauthoryear{Haddock et~al.}{2022}]{haddock2022quantile}
\begin{barticle}
\bauthor{\bsnm{Haddock}, \binits{J.}},
\bauthor{\bsnm{Needell}, \binits{D.}},
\bauthor{\bsnm{Rebrova}, \binits{E.}},
\bauthor{\bsnm{Swartworth}, \binits{W.}}:
\batitle{Quantile-based iterative methods for corrupted systems of linear
  equations}.
\bjtitle{SIAM Journal on Matrix Analysis and Applications}
\bvolume{43}(\bissue{2}),
\bfpage{605}--\blpage{637}
(\byear{2022})
\doiurl{10.1137/21M1429187}
\end{barticle}
\endbibitem

\bibitem[\protect\citeauthoryear{Steinerberger}{2023}]{steinerberger2023quantile}
\begin{barticle}
\bauthor{\bsnm{Steinerberger}, \binits{S.}}:
\batitle{Quantile-based random {K}aczmarz for corrupted linear systems of
  equations}.
\bjtitle{Information and Inference: A Journal of the IMA}
\bvolume{12}(\bissue{1}),
\bfpage{448}--\blpage{465}
(\byear{2023})
\doiurl{10.1093/imaiai/iaab029}
\end{barticle}
\endbibitem

\bibitem[\protect\citeauthoryear{Tondji and Lorenz}{2023}]{tondji2023faster}
\begin{barticle}
\bauthor{\bsnm{Tondji}, \binits{L.}},
\bauthor{\bsnm{Lorenz}, \binits{D.A.}}:
\batitle{Faster randomized block sparse {K}aczmarz by averaging}.
\bjtitle{Numerical Algorithms}
\bvolume{93}(\bissue{4}),
\bfpage{1417}--\blpage{1451}
(\byear{2023})
\end{barticle}
\endbibitem

\bibitem[\protect\citeauthoryear{Merzlyakov}{1963}]{merzlyakov1963relaxation}
\begin{barticle}
\bauthor{\bsnm{Merzlyakov}, \binits{Y.I.}}:
\batitle{On a relaxation method of solving systems of linear inequalities}.
\bjtitle{USSR Computational Mathematics and Mathematical Physics}
\bvolume{2}(\bissue{3}),
\bfpage{504}--\blpage{510}
(\byear{1963})
\doiurl{10.1016/0041-5553(63)90463-4}
\end{barticle}
\endbibitem

\bibitem[\protect\citeauthoryear{Necoara}{2019}]{necoara2019faster}
\begin{barticle}
\bauthor{\bsnm{Necoara}, \binits{I.}}:
\batitle{Faster randomized block {K}aczmarz algorithms}.
\bjtitle{SIAM Journal on Matrix Analysis and Applications}
\bvolume{40}(\bissue{4}),
\bfpage{1425}--\blpage{1452}
(\byear{2019})
\doiurl{10.1137/19M1251643}
\end{barticle}
\endbibitem

\bibitem[\protect\citeauthoryear{Karczmarz}{1937}]{karczmarz1937angenaherte}
\begin{botherref}
\oauthor{\bsnm{Karczmarz}, \binits{S.}}:
Angenaherte auflosung von systemen linearer glei-chungen.
Bull. Int. Acad. Pol. Sic. Let., Cl. Sci. Math. Nat.,
355--357
(1937)
\end{botherref}
\endbibitem

\bibitem[\protect\citeauthoryear{Hounsfield}{1973}]{hounsfield1973computerized}
\begin{barticle}
\bauthor{\bsnm{Hounsfield}, \binits{G.N.}}:
\batitle{Computerized transverse axial scanning (tomography): Part 1.
  description of system}.
\bjtitle{The British journal of radiology}
\bvolume{46}(\bissue{552}),
\bfpage{1016}--\blpage{1022}
(\byear{1973})
\doiurl{10.1259/0007-1285-46-552-1016}
\end{barticle}
\endbibitem

\bibitem[\protect\citeauthoryear{Neumann}{1950}]{neumann1950functional}
\begin{barticle}
\bauthor{\bsnm{Neumann}, \binits{J.V.}}:
\batitle{Functional operators, vol. ii. the geometry of orthogonal spaces (this
  is a reprint of mimeographed lecture notes first distributed in 1933) annals
  of math}.
\bjtitle{Studies Nr. 22 Princeton Univ. Press}
(\byear{1950})
\doiurl{10.1515/9781400882250}
\end{barticle}
\endbibitem

\bibitem[\protect\citeauthoryear{Halperin}{1962}]{halperin1962product}
\begin{barticle}
\bauthor{\bsnm{Halperin}, \binits{I.}}:
\batitle{The product of projection operators}.
\bjtitle{Acta Sci. Math.(Szeged)}
\bvolume{23}(\bissue{1}),
\bfpage{96}--\blpage{99}
(\byear{1962})
\end{barticle}
\endbibitem

\bibitem[\protect\citeauthoryear{Deutsch and Hundal}{1997}]{deutsch1997rate}
\begin{barticle}
\bauthor{\bsnm{Deutsch}, \binits{F.}},
\bauthor{\bsnm{Hundal}, \binits{H.}}:
\batitle{The rate of convergence for the method of alternating projections,
  ii}.
\bjtitle{Journal of Mathematical Analysis and Applications}
\bvolume{205}(\bissue{2}),
\bfpage{381}--\blpage{405}
(\byear{1997})
\doiurl{10.1006/jmaa.1997.5202}
\end{barticle}
\endbibitem

\bibitem[\protect\citeauthoryear{Gal{\'a}ntai}{2005}]{galantai2005rate}
\begin{barticle}
\bauthor{\bsnm{Gal{\'a}ntai}, \binits{A.}}:
\batitle{On the rate of convergence of the alternating projection method in
  finite dimensional spaces}.
\bjtitle{Journal of mathematical analysis and applications}
\bvolume{310}(\bissue{1}),
\bfpage{30}--\blpage{44}
(\byear{2005})
\doiurl{10.1016/j.jmaa.2004.12.050}
\end{barticle}
\endbibitem

\bibitem[\protect\citeauthoryear{Strohmer and
  Vershynin}{2009}]{strohmer2009randomized}
\begin{barticle}
\bauthor{\bsnm{Strohmer}, \binits{T.}},
\bauthor{\bsnm{Vershynin}, \binits{R.}}:
\batitle{A randomized {K}aczmarz algorithm with exponential convergence}.
\bjtitle{Journal of Fourier Analysis and Applications}
\bvolume{15}(\bissue{2}),
\bfpage{262}--\blpage{278}
(\byear{2009})
\doiurl{10.1007/s00041-008-9030-4}
\end{barticle}
\endbibitem

\bibitem[\protect\citeauthoryear{Lorenz et~al.}{2014}]{lorenz2014linearized}
\begin{barticle}
\bauthor{\bsnm{Lorenz}, \binits{D.A.}},
\bauthor{\bsnm{Schopfer}, \binits{F.}},
\bauthor{\bsnm{Wenger}, \binits{S.}}:
\batitle{The linearized {B}regman method via split feasibility problems:
  analysis and generalizations}.
\bjtitle{SIAM Journal on Imaging Sciences}
\bvolume{7}(\bissue{2}),
\bfpage{1237}--\blpage{1262}
(\byear{2014})
\doiurl{10.1137/130936269}
\end{barticle}
\endbibitem

\bibitem[\protect\citeauthoryear{Chen et~al.}{2001}]{chen2001atomic}
\begin{barticle}
\bauthor{\bsnm{Chen}, \binits{S.S.}},
\bauthor{\bsnm{Donoho}, \binits{D.L.}},
\bauthor{\bsnm{Saunders}, \binits{M.A.}}:
\batitle{Atomic decomposition by basis pursuit}.
\bjtitle{SIAM review}
\bvolume{43}(\bissue{1}),
\bfpage{129}--\blpage{159}
(\byear{2001})
\doiurl{10.1137/S003614450037906X}
\end{barticle}
\endbibitem

\bibitem[\protect\citeauthoryear{Cai et~al.}{2009}]{cai2009convergence}
\begin{barticle}
\bauthor{\bsnm{Cai}, \binits{J.-F.}},
\bauthor{\bsnm{Osher}, \binits{S.}},
\bauthor{\bsnm{Shen}, \binits{Z.}}:
\batitle{Convergence of the linearized {B}regman iteration for $l_1$-norm
  minimization}.
\bjtitle{Mathematics of Computation}
\bvolume{78}(\bissue{268}),
\bfpage{2127}--\blpage{2136}
(\byear{2009})
\doiurl{10.1090/S0025-5718-09-02242-X}
\end{barticle}
\endbibitem

\bibitem[\protect\citeauthoryear{Petra}{2015}]{petra2015randomized}
\begin{barticle}
\bauthor{\bsnm{Petra}, \binits{S.}}:
\batitle{Randomized sparse block {K}aczmarz as randomized dual block-coordinate
  descent}.
\bjtitle{Analele {\c{s}}tiin{\c{t}}ifice ale Universit{\u{a}}{\c{t}}ii Ovidius
  Constan{\c{t}}a. Seria Matematic{\u{a}}}
\bvolume{23}(\bissue{3}),
\bfpage{129}--\blpage{149}
(\byear{2015})
\doiurl{10.1515/auom-2015-0052}
\end{barticle}
\endbibitem

\bibitem[\protect\citeauthoryear{Jiang et~al.}{2020}]{jiang2020kaczmarz}
\begin{botherref}
\oauthor{\bsnm{Jiang}, \binits{Y.}},
\oauthor{\bsnm{Wu}, \binits{G.}},
\oauthor{\bsnm{Jiang}, \binits{L.}}:
A {K}aczmarz method with simple random sampling for solving large linear
  systems.
arXiv preprint arXiv:2011.14693
(2020)
\end{botherref}
\endbibitem

\bibitem[\protect\citeauthoryear{Needell and Tropp}{2014}]{needell2014paved}
\begin{barticle}
\bauthor{\bsnm{Needell}, \binits{D.}},
\bauthor{\bsnm{Tropp}, \binits{J.A.}}:
\batitle{Paved with good intentions: analysis of a randomized block {K}aczmarz
  method}.
\bjtitle{Linear Algebra and its Applications}
\bvolume{441},
\bfpage{199}--\blpage{221}
(\byear{2014})
\doiurl{10.1016/j.laa.2012.12.022}
\end{barticle}
\endbibitem

\bibitem[\protect\citeauthoryear{Moorman et~al.}{2021}]{moorman2021randomized}
\begin{barticle}
\bauthor{\bsnm{Moorman}, \binits{J.D.}},
\bauthor{\bsnm{Tu}, \binits{T.K.}},
\bauthor{\bsnm{Molitor}, \binits{D.}},
\bauthor{\bsnm{Needell}, \binits{D.}}:
\batitle{Randomized {K}aczmarz with averaging}.
\bjtitle{BIT Numerical Mathematics}
\bvolume{61}(\bissue{1}),
\bfpage{337}--\blpage{359}
(\byear{2021})
\doiurl{10.1007/s10543-020-00824-1}
\end{barticle}
\endbibitem

\bibitem[\protect\citeauthoryear{Miao and Wu}{2022}]{miao2022greedy}
\begin{barticle}
\bauthor{\bsnm{Miao}, \binits{C.-Q.}},
\bauthor{\bsnm{Wu}, \binits{W.-T.}}:
\batitle{On greedy randomized average block {K}aczmarz method for solving large
  linear systems}.
\bjtitle{Journal of Computational and Applied Mathematics}
\bvolume{413},
\bfpage{114372}
(\byear{2022})
\doiurl{10.1016/j.cam.2022.114372}
\end{barticle}
\endbibitem

\bibitem[\protect\citeauthoryear{Yin}{2010}]{yin2010analysis}
\begin{barticle}
\bauthor{\bsnm{Yin}, \binits{W.}}:
\batitle{Analysis and generalizations of the linearized bregman method}.
\bjtitle{SIAM Journal on Imaging Sciences}
\bvolume{3}(\bissue{4}),
\bfpage{856}--\blpage{877}
(\byear{2010})
\doiurl{10.1137/090760350}
\end{barticle}
\endbibitem

\bibitem[\protect\citeauthoryear{Haddock and
  Needell}{2019}]{haddock2019randomized}
\begin{barticle}
\bauthor{\bsnm{Haddock}, \binits{J.}},
\bauthor{\bsnm{Needell}, \binits{D.}}:
\batitle{Randomized projection methods for linear systems with arbitrarily
  large sparse corruptions}.
\bjtitle{SIAM Journal on Scientific Computing}
\bvolume{41}(\bissue{5}),
\bfpage{19}--\blpage{36}
(\byear{2019})
\doiurl{10.1137/18M1179213}
\end{barticle}
\endbibitem

\bibitem[\protect\citeauthoryear{Cheng et~al.}{2022}]{cheng2022block}
\begin{barticle}
\bauthor{\bsnm{Cheng}, \binits{L.}},
\bauthor{\bsnm{Jarman}, \binits{B.}},
\bauthor{\bsnm{Needell}, \binits{D.}},
\bauthor{\bsnm{Rebrova}, \binits{E.}}:
\batitle{On block accelerations of quantile randomized kaczmarz for corrupted
  systems of linear equations}.
\bjtitle{Inverse Problems}
\bvolume{39}(\bissue{2}),
\bfpage{024002}
(\byear{2022})
\doiurl{10.1088/1361-6420/aca78a}
\end{barticle}
\endbibitem

\bibitem[\protect\citeauthoryear{Zhang et~al.}{2022}]{zhang2022quantile}
\begin{botherref}
\oauthor{\bsnm{Zhang}, \binits{L.}},
\oauthor{\bsnm{Wang}, \binits{H.}},
\oauthor{\bsnm{Zhang}, \binits{H.}}:
Quantile-based random sparse {K}aczmarz for corrupted, noisy linear inverse
  systems.
arXiv preprint arXiv:2206.07356
(2022)
\end{botherref}
\endbibitem

\bibitem[\protect\citeauthoryear{Hansen}{2007}]{hansen2007regularization}
\begin{barticle}
\bauthor{\bsnm{Hansen}, \binits{P.C.}}:
\batitle{Regularization tools version 4.0 for matlab 7.3}.
\bjtitle{Numerical algorithms}
\bvolume{46}(\bissue{2}),
\bfpage{189}--\blpage{194}
(\byear{2007})
\doiurl{10.1007/s11075-007-9136-9}
\end{barticle}
\endbibitem

\bibitem[\protect\citeauthoryear{Polyak}{1964}]{polyak1964some}
\begin{barticle}
\bauthor{\bsnm{Polyak}, \binits{B.T.}}:
\batitle{Some methods of speeding up the convergence of iteration methods}.
\bjtitle{Ussr computational mathematics and mathematical physics}
\bvolume{4}(\bissue{5}),
\bfpage{1}--\blpage{17}
(\byear{1964})
\doiurl{10.1016/0041-5553(64)90137-5}
\end{barticle}
\endbibitem

\end{thebibliography}

\end{document}